\documentclass{article}
\usepackage{amsmath, amssymb, amsthm, tikz, enumitem, longtable, subfig}
\usepackage[colorlinks = true]{hyperref}
\usepackage[all]{xy}
\CompileMatrices
\usepackage[normalem]{ulem}

\newtheorem{procedure}{Procedure}
\newtheorem{definition}{Definition}
\newtheorem{example}{Example}
\newtheorem{theorem}{Theorem}

\DeclareMathOperator{\vc}{vec}

\DeclareMathOperator{\diag}{diag}
\DeclareMathOperator{\nbd}{nbd}

\title{Construction of cospectral graphs}
\author{Supriyo Dutta \\ Department of Mathematics \\ Indian Institute of Technology Jodhpur\\ Email: \texttt{dosupriyo@gmail.com} \vspace{.25 cm} \\ Bibhas Adhikari \\ Department of Mathematics \\ Indian Institute of Technology Kharagpur\\ Email: \texttt{bibhas@maths.iitkgp.ernet.in}}
\date{}

\begin{document} 
\maketitle

\begin{abstract}
	Construction of non-isomorphic cospectral graphs is a nontrivial problem in spectral graph theory specially for large graphs. In this paper, we establish that graph theoretical partial transpose of a graph is a potential tool to create non-isomorphic cospectral graphs by considering a graph as a clustered graph.
\end{abstract}

	\section{Introduction}

		In this work we propose methods to construct cospectral graphs by utilizing the framework of clustered graphs. A graph $G=(V, E)$ with labellings on vertices is called clustered (also called partitioned) if the vertex set is partitioned as $V=\sqcup_{j=1}^n C_j$ where each $C_j$, along with the edge set induced by the vertices in it, is called a cluster of the graph. Throughout the paper, we denote $C_j$ for both as a set of vertices, and the graph (cluster) induced by the vertices in $C_j$; and the meaning will be clarified from the context. For example, a multi-partite graph can be considered as a clustered graph under specific labellings of vertices.

		The adjacency matrix $\textbf{A}(G)=[a_{ij}]$ associated with a graph $G$ is defined by $a_{ij}=1$ if the vertices $i,j$ are adjacent and $a_{ij}=0$ otherwise. The spectrum of $G$ is the multiset of eigenvalues of $\textbf{A}(G)$. Two isomorphic graphs have the equal spectrum as the corresponding adjacency matrices are permutation similar, but the converse need not be true \cite[Chapter $6$]{cvetkovic1980spectra}. Characterization of graphs that are determined by their spectra is an open problem in algebraic graph theory \cite{dam03which} \cite{OP}. Indeed, in the quest of finding graphs that are determined by their spectra, it is equally important to determine graphs which are not determined by their spectra. Besides, new constructions of cospectral non-isomorphic graphs can have implications on the complexity of the graph isomorphism problem. This calls for developing methods for detection and/or generation of cospectral non-isomorphic graphs. Well-know methods in this directions are the Seidel Switching, Godsil-McKay (GM) switching etc., and many more \cite{dam03which}, \cite{harary1971cospectral}, \cite{seidel1974graphs}, \cite{godsil1982constructing}, \cite{haemers2012seidel}. In this paper we show that partial transpose of a graph developed in \cite{severini2008}, \cite{wu2006conditions} \cite{dutta2016bipartite} can become a handy tool to construct cospectral clustered graphs, in particular large cospectral graphs.

		First we recall the following definition from \cite{dutta2016bipartite}.
		
		\begin{definition}\label{def:1}
			Let $G$ be a clustered graph on $mn$ vertices with clusters $C_i=\{v_{i,1}, v_{i,2}, \dots v_{i,m}\}, i=1,\hdots, n$. Then the \textit{graph theoretical partial transpose} (GTPT) of $G$ is the graph $G^{\tau}$ obtained from $G$ by removing the edges $(v_{i,k}, v_{j,l}),$ for all  $k \neq l, i \neq j$ in $G$ and correspondingly adding the edges $(v_{i,l},v_{j,k}).$ 
		\end{definition}

		Note that, GTPT is defined for clustered graphs when its clusters contain same number of vertices. Also, the number of edges in $G$ and $G^\tau$ are equal. The adjacency matrix of such a clustered graph $G$ (as defined in Definition \ref{def:1}) can be represented by the block matrix
		\begin{equation}\label{bmatrix}
			\textbf{A}(G) = \begin{bmatrix}
			A_{1,1} & A_{1,2} & \dots & A_{1,n} \\
			A_{2,1} & A_{2,2} & \dots & A_{2,n} \\
			\vdots & \vdots & \vdots & \vdots\\
			A_{n,1} & A_{n,2} &\dots &A_{n,n}
			\end{bmatrix}_{mn \times mn},
		\end{equation} 
		where $A_{ii}$ is the adjacency matrix of the cluster $C_i$ and the block matrix $A_{i,j}$ represents the adjacency relations between the vertices of $C_i$ and $C_j, i\neq j; i,j=1,\hdots, n.$ It is easy to ascertain that the adjacency matrix associated with $G^{\tau}$ is given by $\textbf{A}(G^{\tau})=\textbf{A}(G)^{\tau} = [A_{i,j}^t]_{mn \times mn}$, where $^t$ denotes the transpose of a matrix. The matrix $\textbf{A}^{\tau}$ is called the partial transpose of $\textbf{A}.$  The concept of partial transpose has found many applications in quantum information theory, in particular in the detection of quantum entanglement \cite{peres1996}, \cite{horodecki1997}. If $G$ is isomorphic to $G^\tau$ where the identity map acts as the isomorphism, $G$ is called a partially symmetric graph which is shown to be useful in quantum information theory \cite{dutta2016bipartite}. 

		The graphs $G$ and $G^\tau$ are called GTPT equivalent. Two pertinent questions about the GTPT equivalent graphs are as follows. Are the GTPT equivalent graphs isomorphic and/or cospectral? Are the graphs $G^\tau$ and $H^\tau$ isomorphic if $G$ and $H$ are isomorphic?
		
		The following examples depict that GTPT operation can produce isomorphic graphs, non-isomorphic and non-cospectral graphs, and non-isomorphic but cospectral graphs. For instance the graphs $G_1$ and $G_1^\tau$ given by 
		$$G_1 = \xymatrix{\bullet_{1,1} \ar@{-}[d] \ar@{-}[dr] & \bullet_{1,2} \ar@{-}[d] \\ \bullet_{2,1} & \bullet_{2,2}} \hspace{2cm} G_1^\tau = \xymatrix{\bullet_{1,1} \ar@{-}[d] & \bullet_{1,2} \ar@{-}[d] \ar@{-}[dl] \\ \bullet_{2,1} & \bullet_{2,2}}$$ 
		are GTPT equivalent and isomorphic, hence cospectral. The graphs $G_2$ and $G_2^\tau$ given by
		$$G_2 = \xymatrix{\bullet_{1,1} \ar@{-}[r] & \bullet_{1,2} \ar@{-}[d] \ar@{-}[dl] \\ \bullet_{2,1} & \bullet_{2,2}} \hspace{2cm} G_2^\tau = \xymatrix{\bullet_{1,1} \ar@{-}[r] \ar@{-}[dr] & \bullet_{1,2} \ar@{-}[d] \\ \bullet_{2,1} & \bullet_{2,2}}$$ 
		are GTPT equivalent but neither isomorphic nor cospectral. Indeed the GTPT equivalent graphs $G_3$ and $G_3^\tau$ given by 
		$$G_3=\xymatrix{\bullet_{1,1} \ar@{-}[d] \ar@{-}[dr] & \bullet_{1,2} \ar@{-}[d] & \bullet_{1,3} \ar@{-}[d] \ar@{-}[dl] \\ \bullet_{2,1} \ar@{-}[r]& \bullet_{2,2} & \bullet_{2,3} \ar@{-}[l]} \hspace{2cm} G_3^\tau=\xymatrix{\bullet_{1,1} \ar@{-}[d] & \bullet_{1,2} \ar@{-}[d] \ar@{-}[dr] \ar@{-}[dl] & \bullet_{1,3} \ar@{-}[d] \\ \bullet_{2,1} \ar@{-}[r] & \bullet_{2,2} & \bullet_{2,3} \ar@{-}[l]}$$ 
		are non-isomorphic but cospectral. The following example shows that two isomorphic graphs $G$ and $H$ need not imply the isomorphism of $G^\tau$ and $H^\tau.$ 
		\begin{example}\label{noncos}
			Consider the graphs $G$ and $H$ as follows.
			$$G = \xymatrix{\bullet_{1,1} \ar@{-}[rd] & \bullet_{1,2} & \bullet_{1,3} \ar@{-}[ld] \\ \bullet_{2,1} \ar@{-}[r] & \bullet_{2,2} \ar@{-}[r] & \bullet_{2,3}} \hspace{1cm} H = \xymatrix{\bullet_{1,1} \ar@{-}[rd] & \bullet_{1,2} \ar@{-}[d] & \bullet_{1,3} \\ \bullet_{2,1} \ar@{-}[r] & \bullet_{2,2} \ar@{-}[r] & \bullet_{2,3}}$$
			Then 
			$$G^\tau = \xymatrix{\bullet_{1,1} & \bullet_{1,2} \ar@{-}[rd] \ar@{-}[ld] & \bullet_{1,3} \\ \bullet_{2,1} \ar@{-}[r] & \bullet_{2,2} \ar@{-}[r] & \bullet_{2,3}} \hspace{1cm} H^\tau = \xymatrix{\bullet_{1,1} & \bullet_{1,2} \ar@{-}[d] \ar@{-}[ld] & \bullet_{1,3} \\ \bullet_{2,1} \ar@{-}[r] & \bullet_{2,2} \ar@{-}[r] & \bullet_{2,3}}$$
			are not isomorphic. Computing the eigenvalues of $G^\tau$ and $H^\tau$ it can be easily verified that they are not co-spectral.
		\end{example}

		Several approches are introduced in literature for the construction of co-spectral graphs. For example, see \cite{cvetkovic1988recent}, \cite{fujii1999isospectral}, \cite{halbeisen1999generation}, \cite{rowlinson1996characteristic}, \cite{godsil1982constructing}, \cite{knuth1997partitioned} and the references therein. Schwenk et al. investigated the problem of construction of cospectral graphs for a few structured graphs using the concept of bigraphs and the characteristic polynomial of a graph, for example see \cite{schwenk1979construction}, \cite{schwenk1973almost}. Indeed, observe that these techniques are proposed for graphs with a very specific structure and hence the applicability of these techniques is limited. In this paper, we show that the GTPT approach generates a pair of cospectral graphs from any given bipartite graph, and for a graph on a composite number of verices it opens up a possibility of generation of a cospectral mate of the given graph. Note that given a graph on $n$ vertices, where $n$ is a composite number, multiple clustered graphs can be obtained by multiple factorizations of $n.$ Hence  multiple GTPT equivalent graphs can be obtained for the same graph. Besides, since the GTPT approach acts like a matrix function on the algebra of (weighted) adjacency matrices, the matrix theoretic results can be used to determine specific properties of an adjacency matrix to ensure cospectrality. Here we mention that Willem H. Haemers attempted to investigate the Seidel switching operation as a matrix operation  in \cite{haemers2012seidel,haemers2004enumeration}.
		
		The main contributions of this paper are as follows. We determine classes of graphs for which GTPT approach gurantees to produce cospectral graphs, for example bipartite graphs and pseudo bipartite graphs defined in Section \ref{Sec:2}. By using matrix theoretic arguments we prove that if the blocks of the adjacency matrix of a clustered graph form a set of normal commuting matrices then its GTPT equivalent graph is cospectral. Further, using the recently developed graph structure corresponding to such adjacency matrices we demonstrate the class of graphs which ascertain cospectrality with its GTPT. In Section \ref{Sec:3} we propose procedures to create new GTPT equivalent cospectral graphs by utilizing GTPT equivalent cospectral graphs. These procedures can be used to generate large cospectral graphs. Finally, in Section \ref{Sec:4}, we produce  several GTPT equivalent non-isomorphic cospectral graphs by employing the procedures introduced in Section \ref{Sec:3}. Thus we establish that GTPT approach can act as potential method for generation of non-isomorphic cospectral graphs. Then we conclude this article with some future research problems in this direction.

	\section{Constructing cospectral graphs}\label{Sec:2}
		
		In this section, we determine classes of GTPT equivalent graphs that are cospectral. Before that we recall GM-switching as an opertion on certain type of matrices as described in \cite{haemers2004enumeration}. We also show that partial transpose of the adjacency matrix of a clustered graph bears a partial resemblance of GM-switching.
		
		First we recall the following theorem which explains GM-switching as a matrix operation. Let $\textbf{1}$ denote the all-one vector.
		
		\begin{theorem}\label{thm:gm}
			\cite{haemers2004enumeration} Let $N$ be a $(0, 1)$-matrix of size $b \times c$ (say) whose column sums are $0, b$ or $b/2$. Define $\widetilde{N}$ to be the matrix obtained from $N$ by replacing each column $v$ with $b/2$ ones by its complement $\textbf{1}-v$. Let $B$ be a symmetric $b\times b$ matrix with constant row (and column) sums, and let $C$ be a symmetric $c \times c$ matrix. Put
			$$M = \begin{bmatrix}
			B & N\\ N^t & C
			\end{bmatrix}
			\,\, \mbox{and} \,\, \widetilde{M} = \begin{bmatrix}
			B & \tilde{N}\\ \tilde{N}^t & C
			\end{bmatrix}.$$ Then $M$ and $\widetilde{M}$ are cospectral.
		\end{theorem}
		
		Observe from Theorem \ref{thm:gm} that considering a graph $G$ as a clustered graph with $2$ clusters and $M$ as its adjacency matrix, the matrices $B$ and $C$ represent the adjacency matrices of the clusters and $N$ captures the adjacency relations between vertices in $B$ and $C.$ The switching of $G$ is produced by removing and adding some edges which connect the vertices in $B$ with the vertices in $C.$ Whereas, in contrast to GM-switching, a  graph $G$ can have any number of clusters depending on the number of vertices in $G$ for GTPT and each cluster should have same number of vertices. Alike GM-switching the only removal/addition of edges are done in GTPT only for the edges which link the vertices between any two clusters.
		
		\subsection{Bipartite and pseudo-bipartite graphs} 
		
			Recall that a bipartite graph is a clustered graph with $2$ clusters. Now we define pseudo-bipartite graphs as follows.
			
			\begin{definition}(Pseudo-bipartite graph) 
				Let $G = (V(G), E(G))$ be a clustered graph on $2m$ vertices having two clusters say $C_i=\{v_{i,j} : j=1,\dots, m\}, i=1,2$ each of which contains $m$ vertices. Then $G$ is said to be a pseudo-bipartite graph if $C_1$ and $C_2$ are isomorphic and the `identity' map $Id:V(C_1)\rightarrow V(C_2)$ defined as $Id(v_{1,j})=v_{2,j}$ acts as the isomorphism. 
			\end{definition}
			
			Then we have the following theorem.
			
			\begin{theorem}\label{Thm:psdbi}
				The GTPT equivalent pseudo-bipartite graphs are isomorphic and hence co-spectral. In particular, GTPT equivalent bipartite graphs are isomorphic when the clusters in the bipartite contain same number of vertices.
			\end{theorem}
			
			\begin{proof}
				Let $G$ be a pseudo-bipartite graph on $2m$ vertices where the the clusters $C_1$ and $C_2$ of $G$ are isomorphic and $Id: V(C_1)\rightarrow V(C_2)$ is the isomorphism. Let $C_1=\{v_{1,i}: i=1,\hdots,m\}$ and label the vertices of $C_2$ by $v_{2,i}=Id(v_{1,i}), i=1,\hdots,m.$ Observe that any edge of the type $(v_{i,j}, v_{i,k}), i=1,2$ and $j\neq k$ does not get affected in the formation of $G^\tau.$ Now consider the function $f: V(G)\rightarrow V(G^\tau)$ as defined by \begin{equation}\label{def:f}
				f(v_{i,j}) = \begin{cases} v_{2,j} &~\mbox{for}~ i = 1, \\ v_{1,j} &~\mbox{for}~ i = 2. \end{cases}
				\end{equation} Then $(v_{1,i}, v_{1,j})\in E(G)$ implies $(f(v_{1,i}), f(v_{1,j}))=(v_{2,i}, v_{2,j})=(Id(v_{1,i}), Id(v_{1,j}))\in E(G^\tau)$ since $Id$ is an isomorphism. Similarly, $(v_{2,i}, v_{2,j})\in E(G)$ means $(f(v_{2,i}),$ $ f(v_{2,j})) = (v_{1,i}, v_{1,j})\in E(G^\tau)$ since $(v_{2,i}, v_{2,j}) = (Id(v_{1,i}), Id(v_{1,j})) \in E(G)$ and $\psi$ is isomorphism. For other edges of the form $(v_{1,i}, v_{2,j})$ the result follows as in the case of bipartite graphs. Hence the proof.
			\end{proof}
			
			We mention that the Theorem \ref{Thm:psdbi} is no more true when the pseudo-bipartite graph is replaced by a graph with $2$ isomorphic clusters say $C_1, C_2$ and the isomorphism is not the identity map. For instance consider the GTPT equivalent graphs in Figure \ref{counterexamples} in which spectrum of $G$ is $ \{2,  1.618, -2, -1.618,  0.618, -0.618, 0, 0\}$ and spectrum of $G^\tau$ is $\{-2.1490, -1.5434, 2.149, 1.5434, 0, 0, 0, 0\}.$
			
			\begin{figure}
				\centering
				\subfloat[$G$.]{%
					\resizebox*{5cm}{!}{
						\begin{tikzpicture}[scale = .8]
						\node at (0,0) {$\bullet_{21}$};
						\node at (2,0) {$\bullet_{22}$};
						\node at (4,0) {$\bullet_{23}$};
						\node at (6,0) {$\bullet_{24}$};
						\node at (0,2) {$\bullet_{11}$};
						\node at (2,2) {$\bullet_{12}$};
						\node at (4,2) {$\bullet_{13}$};
						\node at (6,2) {$\bullet_{14}$};
						\draw (-.1, .05) -- (1.9, .05);
						\draw {[rounded corners] (-.1, .05) --(2, -.25) -- (3.9, .05)};
						\draw {[rounded corners] (-.1, .05) --(3, -.6) -- (5.9, .05)};
						\draw (-.1, 2.05) -- (1.9, 2.05);
						\draw (1.9, 2.05) -- (3.9, 2.05);
						\draw {[rounded corners] (1.9, 2.05) -- (3.9, 2.25) -- (5.9, 2.05)};
						\draw (5.9, 2.05) -- (1.9, .05);
						\end{tikzpicture}
				}}\hspace{5pt}
				\subfloat[$G'$.]{%
					\resizebox*{5cm}{!}{
						\begin{tikzpicture}[scale = .8]
						\node at (0,0) {$\bullet_{21}$};
						\node at (2,0) {$\bullet_{22}$};
						\node at (4,0) {$\bullet_{23}$};
						\node at (6,0) {$\bullet_{24}$};
						\node at (0,2) {$\bullet_{11}$};
						\node at (2,2) {$\bullet_{12}$};
						\node at (4,2) {$\bullet_{13}$};
						\node at (6,2) {$\bullet_{14}$};
						\draw (-.1, .05) -- (1.9, .05);
						\draw {[rounded corners] (-.1, .05) -- (2, -.25) -- (3.9, .05)};
						\draw {[rounded corners] (-.1, .05) -- (3, -.6) -- (5.9, .05)};
						\draw (-.1, 2.05) -- (1.9, 2.05);
						\draw (1.9, 2.05) -- (3.9, 2.05);
						\draw {[rounded corners] (1.9, 2.05) -- (3.9, 2.25) -- (5.9, 2.05)};
						\draw (1.9, 2.05) -- (5.9, .05);
						\end{tikzpicture}
				}}
				\caption{The graph $G$ has two isomorphic clusters and $G,$ $G^\tau$ are not cospectral mates} 
				\label{counterexamples}
			\end{figure}
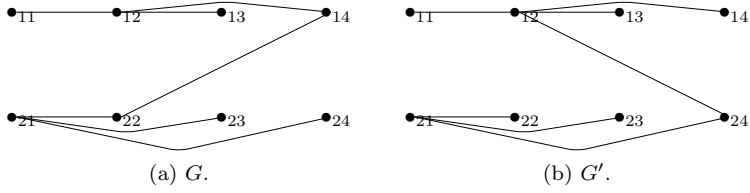
			
			Now we have the following resut for bipartite graphs. 
			\begin{theorem}\label{Thm:bipp}
				Let $G$ be a bipartite graph such that the vertex set is partitioned as $V(G)=C_1\sqcup C_2$ where $|C_1|=m_1\neq m_2=|C_2|.$ Then $G$ and $G^\tau$ are isomorphic.
			\end{theorem}
			\begin{proof}
				Without loss of generality assume that $m_1 < m_2$. Then add $(m_2 - m_1)$ adhoc isolated vertices into the set $C_1$ such that the resultant graph becomes a bipartite graph with equal partition. For brevity we use the same notation for the modified graph as $G$. Then by using Theorem \ref{Thm:psdbi} $G$ and $G^\tau$ are isomorphic and the isomorphism must correspond the isolated vertices in $G$ to isolated vertices in $G^\tau.$ Hence by removing the isolated vertices from $G$ and $G^\tau$ the desired result follows.
			\end{proof}
			
		\subsection{Construction of cospectral GTPT equivalent graphs}
			
			As mentioned in the introduction recall that two GTPT equivalent graphs need not be cospectral. In this subsection we provide a sufficient condition on the stuructual properties of a clustered graph $G$ so that the graph $G^\tau$ becomes cospectral with $G.$ First we prove the following theorem.
			
			\begin{theorem}\label{Thm:solution}
				Let $\{A_i, i=1,\hdots, k\}$ be a commuting family of normal  real matrices of order $m$. Then there exists a nonsingular matrix $X$ such that $A_i^t = X^{-1}A_iX$ for $i=1,\hdots, k.$ 
			\end{theorem}
			
			\begin{proof} 
				Note that the problem of finding a nonsinglar matrix $X$ for which $A_i^t = X^{-1}A_iX$ is equivalent to solving a syetem of Lyapunov equations $XA_i^T=A_iX, i=1,\hdots, k.$ This implies that it is enough to find a nontrivial solution for the set of linear systems of the form 
				\begin{equation}\label{Kroneqn} 
					(I_m\otimes A_i -A_i^t\otimes I_m)x=0, \,\,  i=1,\hdots, k 
				\end{equation} 
				such that $x=\vc(X)=[X_1^t \, X_2^t \, \hdots, X_m^t]^t,$ the vectorizaton of some nonsingular matrix $X$ whose $j$th column is $X_j, j=1,\hdots, m$, and $I_m$ is the identity matrix of order $m$.
				
				Recall that a commuting family of normal matrices is simultaneously unitarily diagonalizable [Theorem 2.5.5, \cite{horn2012matrix}], that is, there exists a unitary matrix $U$ such that $$A_i=U^*D_iU$$ where $D_i$ is a diagonal matrix (not necessarily real) and $U^*$ denotes the conjugate transpose of $U.$ Hence $A_i^t=U^tD_i\overline{U}, i=1,\hdots, k.$ Let $S(A_i)=I_m \otimes A_i - A_i^t \otimes I_m.$ Then
				\begin{eqnarray}
					S(A_i) &=& I_m \otimes U^* D_i U - U^t D_i \overline{U} \otimes I_m \nonumber \\
					& = & (U^t I_m \overline{U})\otimes (U^* D_i U) - (U^t D_i \overline{U}) \otimes (U^* I_m U) \nonumber \\
					& =& (U^t \otimes U^*) (I \otimes D_i - D_i \otimes I) (\overline{U} \otimes U) \nonumber\\
					&=& (\overline{U} \otimes U)^* (I \otimes D_i - D_i \otimes I) (\overline{U} \otimes U)
				\end{eqnarray} 
				where $\overline{U}=[\overline{u}_{ij}]$ if $U=[u_{ij}].$ 
				
				Let $\mathcal{U} = (\overline{U} \otimes U)$. Then  $\mathcal{U} \mathcal{U}^* = (\overline{U} \otimes U)(U^t \otimes U^*) = (\overline{U} U^t) \otimes (U U^*) = I_{m^2},$ that is $\mathcal{U}$ is a unitary matrix. Then observe that the equation (\ref{Kroneqn}) takes the form, $S(D_i)\mathcal{U} x = 0$. Denoting $y = \mathcal{U} x$ we get, $S(D_i)y = 0, i=1,\dots, k$. 
				
				Note that, $y$ is a vector of order $m^2$. Setting 
				$$\hat{y}_j = (y_{(j - 1)m + 1} , y_{(j - 1)m + 2} , \dots , y_{jm})^t \in \mathbb{R}^m,  j= 1, \dots, m$$ 
				we obtain $y = (\hat{y}_1, \hat{y}_2, \dots \hat{y}_m)^t$. Suppose $D_i = \diag\{\lambda_1^{(i)}, \lambda_2^{(i)}, \dots, \lambda_m^{(i)}\}$. Then $I_m \otimes D_i = \diag\{D_i, D_i, \dots D_i\}$, a block diagonal matrix of order $m$, and $D_i \otimes I_m = \diag\{\lambda_1^{(i)} I_m, \lambda_2^{(i)} I_m, \dots \lambda_m^{(i)} I_m\}$. Observe that in each of the diagonal matrices $S(D_i)$ at least $m$ diagonal entries are zero. They are at the position $1, m + 2, 2m + 3, \dots, m^2$ of the diagonal. Thus $1, m + 2, 2m + 3, \dots, m^2$-th entries of the vector $y$ may be arbitrary scalars and other entris of $y$ chosen to be $0$ provide a solution of $S(D_i)y=0, i=1,\dots, k$. Moreover $\{\hat{y}_1, \hat{y}_2, \dots \hat{y}_m\}$ forms a set of linearly independent vectors. In particular, we may chose $\hat{y}_i = e_i$, the $i$-th column vector of the identity matrix $I_m$.
				
				Finally, 
				$$x=\mathcal{U}^*y=(U^t\otimes U^*)y=\begin{bmatrix} \sum_{j=1}^m u_{j1}U^*\hat{y}_j \\ \vdots \\ \sum_{j=1}^m u_{jm}U^*\hat{y}_j \end{bmatrix}.$$ 
				Set $X_l=\sum_{j=1}^m u_{jl}U^*\hat{y}_j, l=1,\dots, m.$ Then the matrix $X=[X_1 \, X_2 \,\dots \, X_m]$ for which $\vc(X)=x$ is a nonsingular matrix. Indeed we show that $\{X_1, \dots, X_m\}$ forms a set of linearly independent vectors as follows. Let $\alpha_1 X_1 + \alpha_2 X_2 + \dots + \alpha_m X_m =0$ for some scalars $\alpha_1,\dots,\alpha_m.$ This implies
				$$\left(\sum_{j=1}^{m} \alpha_ju_{j1}\right) \hat{y}_1 + \dots + \left(\sum_{j=1}^{m} \alpha_ju_{jm} \right)\hat{y}_m=0$$ 
				which implies $\alpha_j=0, j=1,\dots, m$ since $\{\hat{y}_j : j=1,\dots,m\}$ is a linearly independent set. This completes the proof.
			\end{proof}
			
			Thus, it follows from the above theorem that for a clustered graph $G$ if the blocks $A_{ij}, 1\leq i,j\leq n$ of its adjacency matrix (\ref{bmatrix}) form a set of commuting normal matrices then any GTPT equivalent graph of $G$, for example, $G^\tau$ is a cospectral mate of $G.$ This calls for identifying structural properties of a clustered graph for which this sufficient condition is met. Note that the specific structure of the clusters represented by $A_{i,i}$ (the diagonal blocks of $\textbf{A}(G)$) and the structure of bipartite graphs constituted by the partitioned vertex sets of any two clusters say $C_i, C_j, i\neq j$ and edges between them determine the commuting normality property of $A_{i,j}, 1\leq i,j\leq n.$ We denote $\langle C_i,C_j\rangle$ for the bipartite subgraph of $G$ whose adjacency matrix is given by 
			$$\begin{bmatrix} 0 & A_{i,j} \\ A_{i,j}^t & 0 \end{bmatrix}$$ 
			when $i\neq j.$ If $i=j,$  $\langle C_i,C_i\rangle$ represents the subgraph induced by the $i$th cluster $C_i$ whose adjacency matrix is the diagonal block $A_{i,i}$ of $\textbf{A}(G).$ 

			A graph theoretic characterization of commuting normal blocks of a block matrix is devised in \cite{dutta2017quantum, dutta2017zero} recently. Indeed we summerize the properties obtained in \cite{dutta2017quantum} in the following theorem.

			\begin{theorem}\label{Thm:CommNor}
				Let $G$ be a clustered graph with the clusters $C_i=\{v_{i,1}, v_{i,2}, \dots, v_{i,m}\}$, $i=1,\dots, n.$ For any bipartite graph $\langle C_i, C_j\rangle$ we define  neighborhood index set of a vertex $v_{i,\alpha}$ in $C_i$ with respect to $C_j$ as	
				\begin{equation} 
				\nbd_{C_j}(v_{i,\alpha}) =\{\beta : v_{i,\alpha} \, \mbox{and} \,\,  v_{j,\beta} \, \mbox{are adjacent}\, \}\subseteq \{1,2,\dots,m\}
				\end{equation} 
				where $\alpha,\beta\in\{1,2,\dots, m\}$ and $j\in\{1,2,\dots,n\}.$ Then the blocks of the adjacency matrix $\textbf{A}(G)$  form a set of commuting normal matrices if the following conditions are satisfied. 
				\begin{enumerate} 
					\item 
					(Commuting condition, Theorem 1, \cite{dutta2017quantum}) For any two graphs $\langle C_{i_1}, C_{j_1}\rangle,$ and $\langle C_{i_2}, C_{j_2}\rangle,$ 
					$$\left| \nbd_{C_{j_1}}(v_{i_1, \alpha}) \cap \nbd_{C_{i_2}}(v_{j_2, \beta})\right| = \left|\nbd_{C_{i_1}}(v_{j_1, \beta}) \cap \nbd_{C_{j_2}}(v_{i_2, \alpha})\right|$$ where $1\leq \alpha, \beta \leq m, 1\leq i_1,i_2,j_1,j_2\leq n.$ 
					\item 
					(Normality condition, Theorem 2, \cite{dutta2017quantum}) The number of common neighbors in $C_j$ of any pair of vertices $v_{i,\alpha}, v_{i,\beta}$ in $C_i, i\neq j$ is same as the number of common neighbors in $C_i$ of the pair of vertices $v_{j,\alpha}, v_{j,\beta}$ in $C_j.$
				\end{enumerate}
			\end{theorem} 
			
			Thus we have the following theorem. 
			
			\begin{theorem}\label{Thm:suff}
				Let $G$ be a clustered graph which satisfies the commuting and normality conditions mentioned in Theorem \ref{Thm:CommNor}. Then $G$ and $G^\tau$ are cospectral.
			\end{theorem}
			
			\begin{proof}
				When the subgraphs $\langle C_\mu \rangle$, and $ \langle C_\mu, C_{\nu} \rangle$ satisfy the Theorem \ref{Thm:CommNor}, there is an invertible matrix $P$ such that $A^t_{i, j} = P^{-1}A_{i, j}P$, for all $i$ and $j$ as proved in Theorem \ref{Thm:solution} . Hence we have
				\begin{align*}
					\textbf{A}(G^\tau)  & = 
					\begin{bmatrix}
						A_{1,1}^t & A_{1,2}^t & \dots & A_{1,n}^t \\
						A_{2,1}^t & A_{2,2}^t & \dots & A_{2,n}^t \\
						\vdots & \vdots & \vdots & \vdots\\
						A_{n,1}^t & A_{n,2}^t &\dots &A_{n,n}^t
					\end{bmatrix}
					= \begin{bmatrix}
						P^{-1}A_{1,1}P & P^{-1}A_{1,2}P & \dots & P^{-1}A_{1,n}P \\
						P^{-1}A_{2,1}P & P^{-1}A_{2,2}P & \dots & P^{-1}A_{2,n}P \\
						\vdots & \vdots & \vdots & \vdots\\
						P^{-1}A_{n,1}P & P^{-1}A_{n,2}P &\dots & P^{-1}A_{n,n}P
					\end{bmatrix}\\
					& = \begin{bmatrix} P^{-1} & 0 & \dots & 0 \\ 0 & P^{-1} & \dots & 0 \\ \vdots & \vdots& & \vdots \\ 0 & 0 & \dots & P^{-1} \end{bmatrix} 
					\begin{bmatrix}
						A_{1,1} & A_{1,2} & \dots & A_{1,n} \\
						A_{2,1} & A_{2,2} & \dots & A_{2,n} \\
						\vdots & \vdots & \vdots & \vdots\\
						A_{n,1} & A_{n,2} &\dots &A_{n,n}
					\end{bmatrix}
					\begin{bmatrix} P & 0 & \dots &0 \\ 0 & P & \dots & 0 \\ \vdots & \vdots& & \vdots \\ 0 & 0 & \dots & P \end{bmatrix}\\
					& = \mathcal{P}^{-1}\textbf{A}(G)\mathcal{P}\\
					\text{where,}~ \mathcal{P} & =  \diag\{P, P, \dots P \}.
				\end{align*}
				This completes the proof.	
			\end{proof}
			
			Indeed we emphasize that the condition in Theorem \ref{Thm:suff} is sufficient but not necessary for cospectrality of graphs. This is evident from the following example. 
			
			\begin{example}\label{graphex}
				Consider the following pair of GTPT cospectral graphs:
				$$G = \xymatrix{\bullet_{1,1} \ar@{-}[d] \ar@{-}[dr] & \bullet_{1,2} \\ \bullet_{2,1} \ar@{-}[d] & \bullet_{2,2} \ar@{-}[d] \\ \bullet_{3,1} & \bullet_{3,2}} \hspace{2cm} G^\tau = \xymatrix{\bullet_{1,1} \ar@{-}[d]  & \bullet_{1,2} \ar@{-}[dl] \\ \bullet_{2,1} \ar@{-}[d] & \bullet_{2,2} \ar@{-}[d] \\ \bullet_{3,1} & \bullet_{3,2}}$$
				Then
				$$\textbf{A}(G) = \begin{bmatrix} A_{1,1} & A_{1,2} & A_{1,3} \\ A_{2,1} & A_{2,2} & A_{2,3} \\ A_{3,1} & A_{3,2} & A_{3,3} \end{bmatrix} = \begin{bmatrix} \begin{bmatrix} 0 & 0\\ 0 & 0 \end{bmatrix} & \begin{bmatrix} 1 & 1 \\ 0 & 0\end{bmatrix} & \begin{bmatrix} 0&0\\ 0&0 \end{bmatrix} \\ \begin{bmatrix} 1&0\\ 1&0 \end{bmatrix} & \begin{bmatrix} 0&0\\ 0&0 \end{bmatrix} & \begin{bmatrix} 1&0\\ 0&1 \end{bmatrix}\\ \begin{bmatrix} 0&0\\ 0&0 \end{bmatrix} & \begin{bmatrix} 1&0\\ 0&1 \end{bmatrix} & \begin{bmatrix} 0&0\\ 0&0 \end{bmatrix}\end{bmatrix}.$$
				Recall that any $2\times 2$ matrix is unitarily similar to its transpose (see Lemma 2.4 and Lemma 3.3 in \cite{garcia2009unitary}). Let there be a matrix $P = \begin{bmatrix} a & b \\ c & d \end{bmatrix}$ such that
				$$ A_{1, 2}^t = P^{-1} A_{1, 2} P \Rightarrow \begin{bmatrix} 1 & 0 \\ 1 & 0 \end{bmatrix} = \begin{bmatrix} a & b \\ c & d \end{bmatrix}^{-1} \begin{bmatrix} 1 & 1 \\ 0 & 0 \end{bmatrix} \begin{bmatrix} a & b \\ c & d \end{bmatrix},$$
				$$ A_{2, 1}^t = P^{-1} A_{2, 1} P \Rightarrow \begin{bmatrix} 1 & 1 \\ 0 & 0 \end{bmatrix} = \begin{bmatrix} a & b \\ c & d \end{bmatrix}^{-1} \begin{bmatrix} 1 & 0 \\ 1 & 0 \end{bmatrix} \begin{bmatrix} a & b \\ c & d \end{bmatrix}.$$
				Comparing entrywise we obtain $a = \frac{i}{\sqrt{2}}, b = \frac{i}{\sqrt{2}}, c = \frac{i}{\sqrt{2}}, d = \frac{-i}{\sqrt{2}}$. Thus $P = \frac{i}{\sqrt{2}}\begin{bmatrix} 1 & 1 \\ 1 & -1 \end{bmatrix}$. Hence, $\textbf{A}(G^\tau)=\mathcal{P}^*\textbf{A}(G)\mathcal{P}$, 
				where $\mathcal{P}=\diag\{P, P, P\}$. Consequently  $\textbf{A}(G)$, and $\textbf{A}(G^\tau)$ are co-spectral. Indeed, note that $A_{1,2}A_{2,1}\neq A_{2,1}A_{1,2}.$
			\end{example}
			
			A critical observation from the above example leads to a procedure of constructing cospectral graphs that violate the sufficient condition of Theorem \ref{Thm:suff} is given in the following theorem.
			
			\begin{theorem}\label{nnornal}
				Let $A$ be a non-normal binary matrix of order $m$ and $A$ is similar to its transpose. Consider a clusterd graph $G$ on $mn$ vertices with $n$ clusters $\{C_1,\dots, C_n\}$ each of which contains $m$ vetices. Assume that $G$ has the following structural properties. 
				\begin{enumerate}
					\item
					The induced sub-graph of $G$ defined by $C_i, 1\leq i \leq n$ is a graph with no edges.
					\item
					The adjacency matrix corresponding to the bipartite graphs $\langle C_i \cup C_j\rangle$ is of the form 
					$$\begin{bmatrix} 0_{m\times m} & A_{i,j}\\ A_{i,j}^t &0_{m\times m}\end{bmatrix}$$ 
					and $A_{i,j}\in\{A, I_m, 0_{m\times m}\}, i\neq j$ where $I_m$ is the identity matrix of order $m,$ $0_{m\times m}$ is the zero matrix of order $m,$ and at least for one pair of $(i,j), A_{i,j}=A.$
				\end{enumerate}
				Then the graphs $G$ and $G^\tau$ are co-spectral but the commuting normality property of blocks in $\textbf{A}$ is not satisfied.
			\end{theorem} 
			
			\begin{proof}
				The proof follows from the fact that the matrix $A$ is similar to its transpose. 
			\end{proof}
			
			Here we mention that construction of a non-normal binary matrix is easy. Indeed $r_i=c_i$ where $r_i$ is the sum of the entries of $i$th row and $c_i$ is the sum of entries of $i$th colum of the matrix is a necessary conditon for a binary matrix to be normal. Thus any binary matrix violating this condition is an example of a non-normal matrix and the condition for similarity to its transpose can be found in \cite{garcia2009unitary}. In this context we mention that if the blocks of the adjacency matrix $\textbf{A}(G)$ is normal then the degree sequence of $G$ shall be same as the degree sequence of $G^\tau.$ This indicates a possibility that $G^\tau$ could be isomorphic to $G.$ Whereas, for a non-normal off-diagonal block matrix of $\textbf{A}(G)$ the degree sequence of $G$ and $G^\tau$ need not be equal. Hence Theorem \ref{nnornal} can provide a useful tool for generation of non-isomorphic cospectral graphs.
			
			Obviously the Theorem \ref{nnornal} can be generalized in many ways. For instance, if the matrices $A_{i,j}\in \{A_l, I_m, 0_{m\times m}, l=1,\dots,k\}$ such that there exists a non-singular matrix $X$ for which $X^{-1}A_lX=A_l^t$ and $A_l, l=1,\dots,k$ form a set of non-normal and/or non-commutating matrices then the graph $G^\tau$ is cospectral but need not be isomorphic to $G.$

	\section{Construction of cospectral graphs from GTPT equivalent cospectral graphs}\label{Sec:3}

		In the following we consider a clustered graph $G$ on $mn$ vertices  with clusters $C_i=\{v_{i,1}, v_{i,2}, \dots, v_{i,m}\}$, $i=1,\dots, n$ such that $G$ and $G^\tau$ are cospectral. Denote the adjacency block matrix associated with the graph $G$ as $\textbf{A}(G)=[A_{ij}], 1\leq i,j\leq n.$ We develop procedures for construction of new cospectral graphs from the GTPT equivalent cospectral graphs $G$ and $G^\tau$ as follows.
		
		\begin{procedure}\label{pros1}
			Let $G$ be a clustered graph which inherits the structural properties possessed by the commuting and normality conditions mentioned in Theorem \ref{Thm:CommNor}. Construct a clustered graph $H=(V(H), E(H))$ on $km$ vertices with clusters $D_j, j=1,\dots,k$ defined as follows. 
			\begin{itemize} 
				\item 
				$V(H)=\bigsqcup_{p=1}^k D_p$ where $D_p\in\{C_1, C_2,\dots, C_n\}$ where $k\geq 2.$  
				\item 
				Generate edges between vertices of $D_p$ and $D_q, q\neq p$ such that the adjacency matrix corresponding to the bipartite graph $\langle D_p,D_q\rangle, 1\leq p,q\leq k$ is given by 
				$$\begin{bmatrix} 0_{m\times m}&D_{p,q}\\ D_{p,q}^t & 0_{m\times m}\end{bmatrix}$$ 
				such that $D_{p,q}\in \{I_{m}, 0_{m\times m}\} \cup \{ A_{i,j} | i\neq j, 1\leq i,j\leq n\}.$ Thus either no vertex of $D_p$ is adjacent to any vertex of $D_q,$ or the $l$th vertex $v_{p,l}$ of $D_p$ is adjacent to $l$th vertex $v_{q,l}$ of $D_q$ for all $1\leq l\leq m$ only, or $\langle D_p,D_q\rangle=\langle C_i,C_j\rangle$ for some $i$ and $j\neq i.$  
			\end{itemize}
		\end{procedure}
		
		\begin{theorem}
			The clustured graph $H$ constructed from a clustered graph $G$ using Procedure-1 is a cospectral mate of the graph $H^\tau.$
		\end{theorem}
		\begin{proof}
			The proof follows from the fact that the blocks of adjacency matrix associated with the graph $H$ form a family of similar matrices with a common similarity matrix.
		\end{proof}
		
		Observe that the crux behind the construction of the cospectral graphs $H$ and $H^\tau$ in Procedure-1 is by adding/removing some copies of the existing clusters in $G$ and adding/removing copies of some bipartite graphs between the existing clusters. We describe the next procedure for construction of cospectral graphs by adding new vertices in the existing clusters of $G.$ 
		
		\begin{procedure}\label{pros2} 
			Let $G$ be a clustered graph on $mn$ vertices that inherits the structural properties possessed by the commuting and normality conditions mentioned in Theorem \ref{Thm:CommNor}. Construct a clustered graph $H=(V(H), E(H))$ on $(2m-1)n$ vertices with clusters $D_i, i=1,\dots,n$ defined as follows. 
			\begin{itemize} 
				\item 
				$V(H)=\bigsqcup_{i=1}^n D_i$ where $D_i=C_i \cup \, \widehat{C}_i$ where $\widehat{C}_i=\{v_{i,m}, v_{i,m+1}, \dots, v_{i,2m-1}\}.$ Note that $v_{i,j}, j=m+1,\dots,2m-1, i=1,\dots, n$ are new vertices and $C_i \cap \, \widehat{C}_i=\{v_{i,m}\}.$  
				\item 
				Create an edge $(v_{i,l_2}, v_{j,l_2}), 1\leq l_1,l_2\leq m$ in $H$ if and only if the vertices $v_{i,l_2}, v_{j,l_2}$ are adjacent in $G,$ $1\leq i,j\leq n.$ \item Consider the map $f: \bigsqcup_{i=1}^n C_i \rightarrow \bigsqcup_{i=1}^n \widehat{C}_i$ defined by $f(v_{i,l})=v_{i,2m-l}, l=1, \dots, m, i=1,\dots,n.$ Then define an edge between $f(v_{i,l_1})$ and $f(v_{j,l_2}), 1\leq l_1,l_2\leq m$ if and only if $v_{i,l_1}$ and $v_{j,l_2}$ are adjacent in $G$ for all $i\neq j, 1\leq i,j\leq n.$ 
			\end{itemize}
		\end{procedure}
		
		Some immediate observations about the graph $H$ constructed by using Procedure 2 are as follows. Note that, the vertex set of $H$ can be written as $V(H)=\bigsqcup_{i=1}^n C_i\cup \widehat{C}_i$ and the induced subgraph generated by the vertex set $\sqcup_{i=1}^n C_i$ in $H$ is $G.$  
		
		\begin{theorem}
			The clustured graph $H$ constructed from a clustered graph $G$ using Procedure-2 is a cospectral mate of the graph $H^\tau.$
		\end{theorem}
		
		\begin{proof}
			Consider the subgraph $\langle C_i, C_j \rangle$ of $G$. Then the subgraph $\langle D_i, D_j \rangle$ of $H$ contains all the edges of $\langle C_i, C_j \rangle$ and some additional edges that preserve the structural properties of $\langle C_i, C_j \rangle$. Let the matrix $$\begin{bmatrix} 0_{2m-1\times 2m-1} & B_{i,j}\\ B_{i,j}^t &0_{2m-1\times 2m-1}\end{bmatrix}$$ in $A(H)$ represent the adjacency relations of $\langle D_i, D_j \rangle$ and the $k$-th column of $B_{ij}$ is denoted by $b_{*k}^t$ for $k = 1, 2, \dots (2m - 1)$. Also let the $k$-th column of $A_{ij}$ is denoted by $a_{*k}^t$ for $k = 1, 2, \dots m$. Now, we observe that 
			\begin{equation}
			b_{*k}^t = \begin{cases} [a_{*k} \,\underbrace{0 \, 0 \, \dots 0}_{(m - 1)\text{-times}}]^t & ~\text{for}~ k = 1, 2, \dots (m - 1) \\
			[a_{* m} \, a_{m - 1, m} \, a_{m - 2, m} \, \dots a_{1 ,m}]^t & ~\text{for}~ k = m\\
			[\underbrace{0\, 0\, \dots 0}_{m\text{-times}} a_{m - 1, 2m - k}\, a_{m - 2, 2m - k}\, \dots a_{1, 2m - k}]^t & ~\text{for}~ k = (m + 1), (m + 2), \dots (2m - 1) \end{cases}\\
			\end{equation} 
			
			Note that due to our assumption on $G,$ there is a similarity matrix $P_a=[p_{xy}]$ such that $P_a A^t_{ij} = A_{ij}P_a$ for all $i,j$. Now we show that there is a similarity matrix $P_b$ such that $P_b B^t_{ij} = B_{ij}P_b$ holds for all $i, j,$ and hence  the new graph $H$ will be cospectral to $H^\tau.$ We proceed as follows.
			
			Let the $k$-th columns of $P_a$ and $P_b$ be $p_{*k}^t$ and $\tilde{p}_{*k}^t$ respectively. Also the $k$-th row vector of $A_{ij}$ is $a_{k*}$, that becomes the $k$-th column vector of $A_{ij}^t$ after a transposition. Since $A_{ij}^t = P_a^{-1}A_{ij}P_a$, we obtain $a_{*k}^t = P_a^{-1}A_{ij}p_{*k}^t$. We define
			
			\begin{equation}
			\tilde{p}_{*k}^t  = \begin{cases} [p_{*k} \, \underbrace{0\, 0 \dots 0}_{(m - 1)\text{-times}}]^t & ~\text{for}~ k = 1, 2, \dots (m - 1) \\
			[p_{*m} \, p_{m - 1, m}\, p_{m - 2, m}\, \dots p_{1 ,m}]^t & ~\text{for}~ k = m\\
			[\underbrace{0\, 0\, \dots 0}_{ m\text{-times}} \, p_{m - 1, 2m - k}\, p_{m - 2, 2m - k}\, \dots p_{1, 2m - k}]^t & ~\text{for}~ k = (m + 1), (m + 2), \dots (2m - 1). \end{cases}\\
			\end{equation} 
			Now it is a simple algebraic calculation to verify that $P_b b_{*k}^t = B_{ij}\tilde{p}_{*k}^t$ and $P_b$ is invertible. Hence the proof. 
		\end{proof}
		
		The following example illustrates the relationship between $A_{ij}$ and $B_{ij}$ transparent. 
		\begin{example}
			Consider the graph $G$ of the example \ref{graphex}. Then the graph $H$ generated by applying Procedure 2 on $G$ is given by
			$$H = \xymatrix{\bullet_{1,1} \ar@{-}[d] \ar@{-}[dr] & \bullet_{1,2} & \bullet_{1,3} \ar@{-}[d] \ar@{-}[dl] \\ \bullet_{2,1} \ar@{-}[d] & \bullet_{2,2} \ar@{-}[d] & \bullet_{2,3} \ar@{-}[d] \\ \bullet_{3,1} & \bullet_{3,2} & \bullet_{3,3}}$$
			In the example \ref{graphex} we calculated $A_{1,2} = \begin{bmatrix} 1 & 1 \\ 0 & 0 \end{bmatrix}$ and the block similarity matrix $P_a = \frac{i}{\sqrt{2}} \begin{bmatrix}1 & 1 \\ 1 & -1 \end{bmatrix}$. In $A(H)$, observe that the block $B_{1,2} = \begin{bmatrix} 1 & 1 & 0 \\ 0 & 0 & 0 \\ 0 &1 & -1 \end{bmatrix}$ and the similarity matrix $P_b = \frac{i}{\sqrt{2}} \begin{bmatrix}1 & 1 & 0\\ 1 & -1 & 1 \\ 0 & 1 & 1\end{bmatrix}$. 
		\end{example}
		
		We emphasize that Procedures 1 and 2 can be gainfully used to generate large cospectral graphs and hence compare to GM-switching the idea of GTPT equivalent graphs proves to be more efficient for producing cospectral graphs as almost no cospectral mate of a large graph can be obtained by GM-switching \cite{haemers2004enumeration}. 
		
		Now we recall Example \ref{noncos} which ensures that for two isomorphic (cospectral) graphs $G$ and $H$ the corresponding GTPT equivalent graphs $G^\tau$ and $H^\tau$ need not be cospectral. In the following we introduce a method called alternative clustering of a clustered graph and develop a mechanism for constructing a cospectral GTPT equivalent graph from a given GTPT equivalent cospectral graph.  
		
		For the graph $G$ with clusters $C_i=\{v_{i,1}, v_{i,2}, \dots, v_{i,m}\}$, $i=1,\dots, n,$ we define an altermative clustering on $G$ as $C'_j=\{v_{1,j}, v_{2,j}, \dots, v_{n,j}\}$, $j=1,\dots, m.$ Note that $\sqcup_{i=1}^n C_i = \sqcup_{j=1}^m C'_j.$ Defining these new clusters along with the exiting edges, an isomorphic copy of $G$ is obtained that we call as the alternating clustered graph of $G$ and denote it by $G_{a}.$ Then we have the following theorem. 
		
		\begin{theorem}\label{Th-8} 
			Let $G$ be a clustered graph on $mn$ vertices with $n$ clusters. Then $G^\tau$ and $G_a^\tau$ are isomorphic. Moreover, if $G$ and $G^\tau$ are co-spectral then $G_a$ and $G_a^\tau$ are co-spectral.
		\end{theorem}
		\begin{proof}
			First we prove that $G^\tau$ and $G_a^\tau$ are isomorphic where the clusters of $G$ are given by $C_i=\{v_{i,1}, v_{i,2}, \dots, v_{i,m}\}$, $i=1,\dots, n,$ and the clusters of $G_a$ are described as $C'_j=\{v_{1,j}, v_{2,j}, \dots, v_{n,j}\}$, $j=1,\dots, m.$ Define a bijective map  $f: V(G^\tau) \rightarrow V(G_a^\tau)$ as $f(v_{ij}) = v_{ji}$ for $i = 1, 2, \dots n$ and $j = 1, 2, \dots m$. Now we show that $(u,v)\in E(G^\tau)$ if and only if $(f(u),f(v))\in E(G_a^\tau)$ for any two vertices $u,v\in V(G^\tau).$ We consider the following cases.
			\begin{enumerate}
				\item
				{\bf Case 1:} Let there be an edge $(v_{ij}, v_{ik}) \in E(G)$ for some $j \neq k$. Then $(v_{ji}, v_{ki}) \in E(G_a)$ after the alternative clustering. As GTPT does not change this edge, $(v_{ij}, v_{ik}) \in E(G^\tau)$ and $(v_{ji}, v_{ki}) \in E(G_a^\tau)$, that is, $(f(v_{ij}), f(v_{ik})) \in E(G_a^\tau)$. Further for any edge of the form $(v_{ji}, v_{ki}) \in E(G_a^\tau)$, $(f^{-1}(v_{ji}), f^{-1}(v_{ki})) = (v_{ij}, v_{ik}) \in E(G_a^\tau)$. The proof is similar for any edge $(v_{ij}, v_{kj}) \in E(G)$ for any $j \neq k.$  
				\item
				{\bf Case 2:} Let $(v_{ij}, v_{kl}) \in E(G)$ for some $i \neq k$ and $j \neq l$. Then after alternative clustering on $G$, $(v_{ji}, v_{lk}) \in E(G_a)$. After GTPT on $G$, and $G_a$ we have $(v_{il}, v_{kj}) \in E(G^\tau)$, and $(v_{jk}, v_{li}) \in E(G_a^\tau)$, respectively. As all the graphs are simple, $(v_{jk}, v_{li}) = (v_{li}, v_{jk}) = (f(v_{il}), f(v_{kj}))$. Further, for any edge of the form $(v_{jk}, v_{li}) \in E(G_a^\tau),$ $(f^{-1}(v_{jk}), f^{-1}(v_{li})) = (v_{il}, v_{kj}) \in E(G^\tau)$.
			\end{enumerate} 
			Combining Cases 1 and 2, we prove that $G^\tau$ and $G_a^\tau$ are isomorphic.
			
			Let $\Lambda(X)$ denote the spectrum of a graph $X.$ Then obviously $\Lambda(G)=\Lambda(G_a)$ and $\Lambda(G^\tau)=\Lambda(G_a^\tau)$ since $G$ and $G^\tau$ are isomorphic to $G_a$ and $G_a^\tau$ respectively. Moreover if $G$ and $G^\tau$ are co-spectral, that is, $\Lambda(G)=\Lambda(G^\tau)$ then obviously $\Lambda(G_a)=\Lambda(G_a^\tau).$ This completes the proof.
		\end{proof}
		
	\section{Examples of non-isomorphic cospectral GTPT equivalent graphs}\label{Sec:4}
		
		In order to investigate the efficiency of GTPT technique for constructing non-isomorphic cospectral graphs we attempt to enumerate the number of  non-isomorphic cospectral graphs that can be obtained from a few easily constructable clustered graphs with number of clusters not more than $3.$ In some cases we also employ the Procedures 1 and 2 for the same. Thus in the following we consider clustered model graphs $G$ with clusters $C_i=\{v_{i,1}, v_{i,2}, \dots v_{i,m}\}, i=1,\dots,n, n\leq 3$ and each $C_i$ contains $m\leq 4$ number of vertices. We denote the path and cycle graphs on $m$ vertices as $\mathcal{P}_m$ and $\mathcal{C}_m$ respectively. The number $\kappa$ of non-isomorphic cospectral graphs that can be obtained from each model graph is provided in the following tables. Prototype graphs are generated by using Networkx  \cite{schult2008exploring} which uses the algorithm stated in \cite{cordella2001improved} for checking graph isomorphism. Note that, in all the examples, the clusters are considered as the collection of vertices in each row or column as indicated by the numbers $m$.
		
		First we consider the following model graphs $G$ having clusters $C_i, i=1,\dots, n, n\in\{2, 3\}$ with special structures and assign arbitrary edges to form $\langle C_i, C_j \rangle, i\neq j$ such that the ultimate graph $G$ becomes cospectral and non-isomorphic to $G^\tau.$ We calculate the number of such graphs $G$ possible for each fixed model. 
		\begin{enumerate}[label = Model \arabic*.,itemindent=*]
			\item[1a.] 
			$n=2, m\in\{2,3,4\}.$ The clusters of $G$ are as follows: $C_1$ is a graph with no edges, $C_2=\mathcal{P}_m.$
			\item[1b.]
			$n=2, m\in\{2,3,4\}.$ The clusters of $G$ are as follows: $C_1$ is a graph with no edges, $C_2=\mathcal{C}_m.$
		\end{enumerate}
		
		\begin{longtable}{|p{.02\textwidth}|p{.020\textwidth}|p{.04\textwidth}|p{.80\textwidth}|}
			\hline
			\rotatebox[origin=c]{90}{Model number} & m & $\kappa$ & Examples\\
			\hline
			\hline 
			1a & 2 & 0 & No non-isomorphic GTPT cospectral graph \\
			\cline{2 - 4}
			& 3 & 4 & \\
			\rotatebox[origin=c]{90}{(n = 2)} & & & \includegraphics[scale = .123]{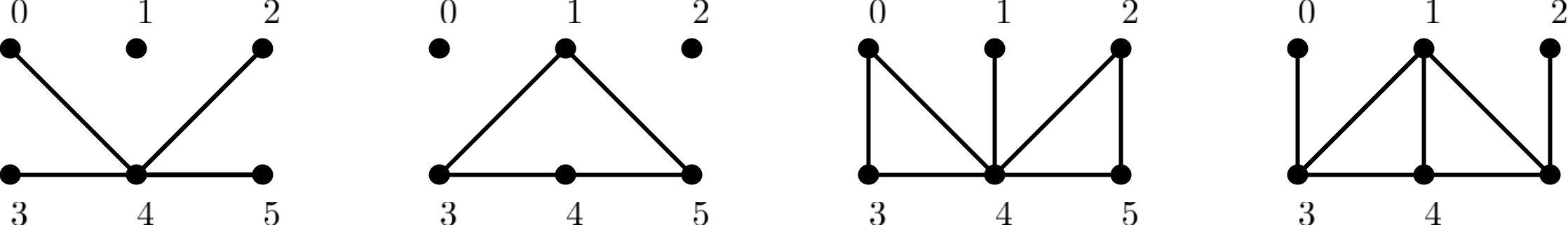}\\
			\cline{2 - 4}
			& 4 & 16 & \\
			& & & \includegraphics[scale = .09]{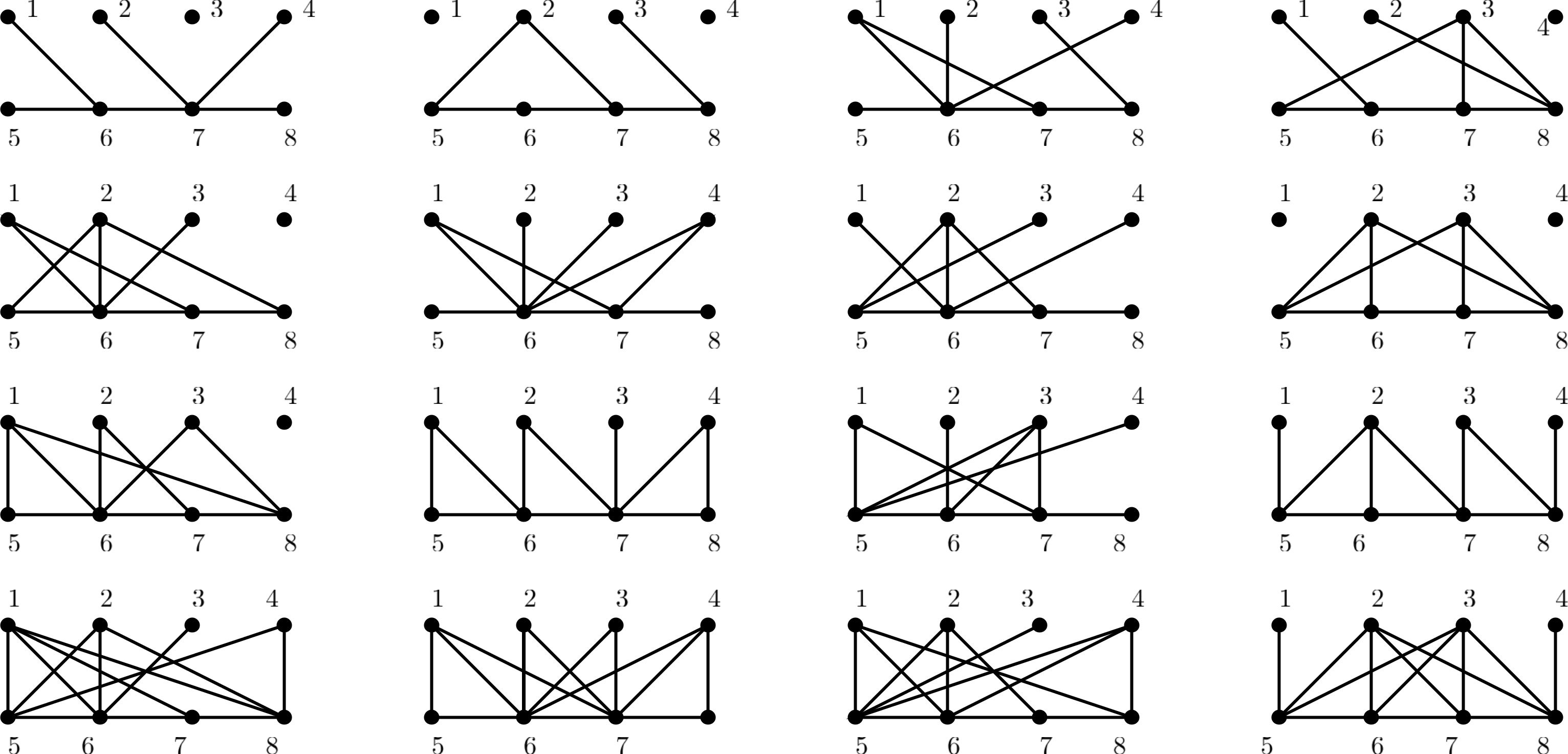}\\
			\hline
			\hline 
			1b& 2 & 0 & No non-isomorphic GTPT cospectral graph\\
			\cline{2 - 4}
			& 3 & 0 & No non-isomorphic GTPT cospectral graph\\
			\cline{2 - 4}
			& 4 & 4 & \\
			\rotatebox[origin=c]{90}{(n = 2)} & & & \includegraphics[scale = .093]{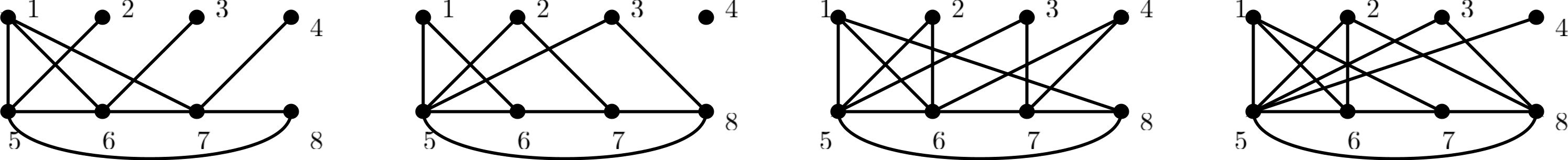}\\
			\hline
		\end{longtable}
		
		Now we consider a few model graphs $G$ that are used to generate graphs $H$ for which $H$ and $H^\tau$ are non-isomorphic and cospectral by applying Procedure 1 or 2 as described below. We calculate the number of non-isomorphic cospectral graphs that can be obtained for each of these model graphs $G$ and draw a few of them in the following table when this number is large.
		
		Examples using Procedure 1.
		
		\begin{enumerate}[label = Model \arabic*.,itemindent=*]
			\item[2.] 
			$n=3$ and $m\in\{2,3,4\}.$ The graph $G$ consistes of $3$ clusters, $C_1, C_2, C_3$ such that the induced subgraph defined by the vertex set $C_1\cup C_2$ is a bipartite graph and $C_3$ is an isolated cluster with no edges. Then construct the graph $H$ by following the Procedure 1 such that $V(H) = V(G)$ and $E(H) = E(G) \cup \{(v_{2,i}, v_{3,i}): i = 1, 2, \dots m\}.$
		\end{enumerate}
		
		\begin{longtable}{|p{.02\textwidth}|p{.020\textwidth}|p{.04\textwidth}|p{.80\textwidth}|}
			\hline
			\rotatebox[origin=c]{90}{Model number} & m & $\kappa$ & Examples\\
			\hline
			\hline 
			2 & 2 & 2 & \\
			& &  & \includegraphics[scale = .1]{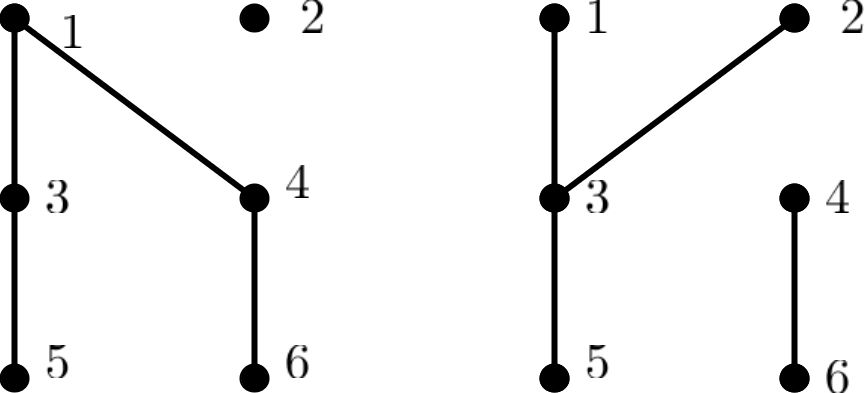}\\
			\cline{2 - 4}
			& 3 & 20 & \\ 
			\rotatebox[origin=c]{90}{(n = 3)} & & & \includegraphics[scale = .083]{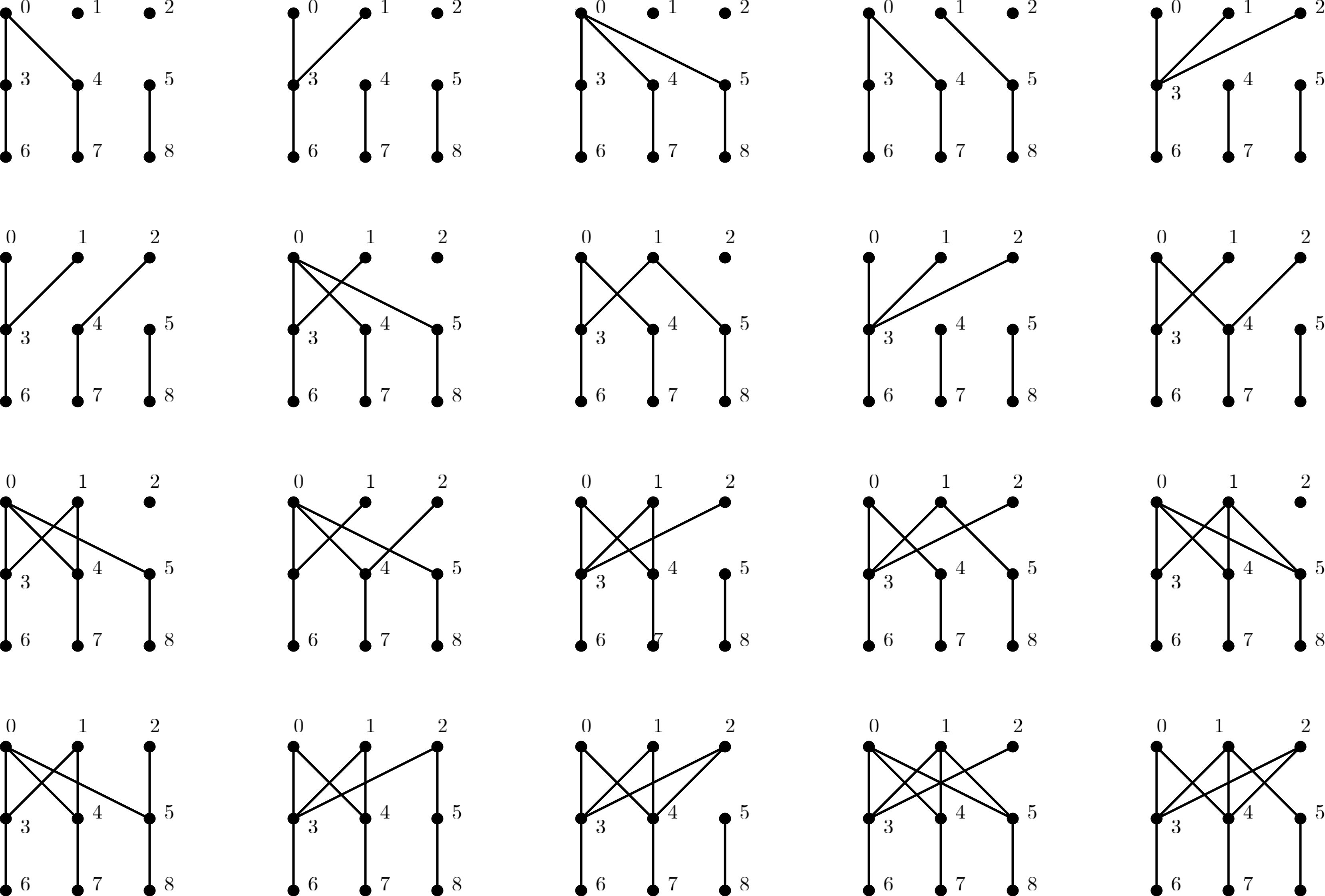}\\
			\cline{2 - 4}
			& 4 & 250 & \\
			& & & \includegraphics[scale = .093]{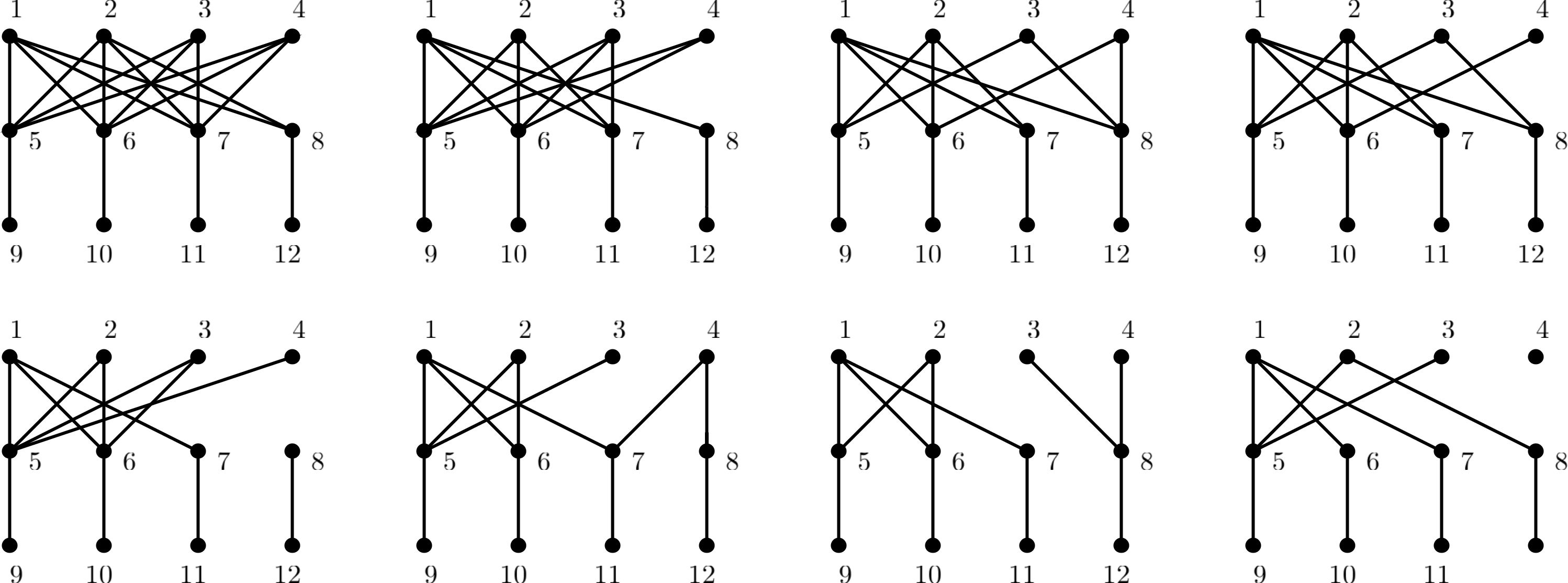}\\
			\hline
		\end{longtable}
		
		Examples using Procedure 2.
		
		\begin{enumerate}[label = Model \arabic*.,itemindent=*]
			\item[3a.] 
			Consider a clustered graph $G$ with $n=2, m\in\{2,3,4\}$ such that the clusters $C_1, C_2$ are graphs with no edges and arbitrary edges between the vertices of $C_1$ and $C_2.$ Consider the alternating clustered graph $G_a$ and apply the Procedure 2 on $G_a$ by adding one vertex in each of the new clusters. Thus the new clustered graph $H$ is obtained with $m=3, n\in\{2,3,4\}.$
			
			\item[3b.]  
			Consider a clustered graph $G$ with $n=2, m\in\{2,3,4\}$ such that the clusters $C_1$ is a graph with no edges, $C_2$ is the path $\mathcal{P}_m$ and arbitrary edges between the vertices of $C_1$ and $C_2.$ Consider the alternating clustered graph $G_a$ and apply the Procedure 2 on $G_a$ by adding one vertex in each of the new clusters. Thus the new clustered graph $H$ is obtained with $m=3, n\in\{2,3,4\}.$
			
			\item[3c.]  
			Consider a clustered graph $G$ with $n=2, m\in\{2,3,4\}$ such that the clusters $C_1$ is a graph with no edges, $C_2$ is the cycle $\mathcal{C}_m$ and arbitrary edges between the vertices of $C_1$ and $C_2.$ Consider the alternating clustered graph $G_a$ and apply the Procedure 2 on $G_a$ by adding one vertex in each of the new clusters. Thus the new clustered graph $H$ is obtained with $m=3, n\in\{2,3,4\}.$
			
		\end{enumerate}
		
		\begin{longtable}{|p{.02\textwidth}|p{.020\textwidth}|p{.020\textwidth}|p{.04\textwidth}|p{.80\textwidth}|}
			\hline
			\rotatebox[origin=c]{90}{Model number} & $m$ & $n$ & $\kappa$ & Examples \\
			\hline
			\hline 
			3a & 3 & 2  & 2 & \\
			& & & & \includegraphics[scale = .09]{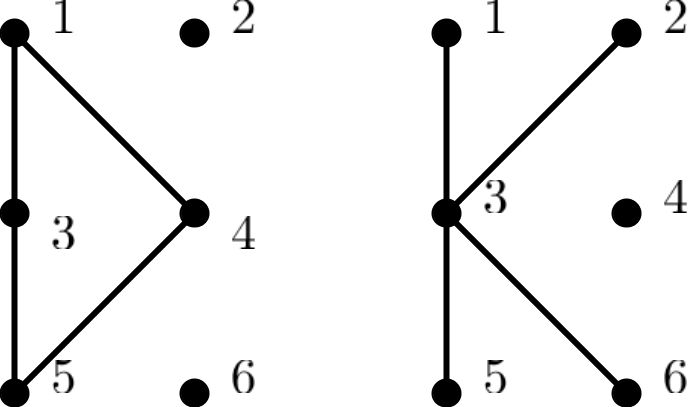}\\
			\cline{2 - 5}
			& 3 & 3 & 20 & \\
			& & & & \includegraphics[scale = .095]{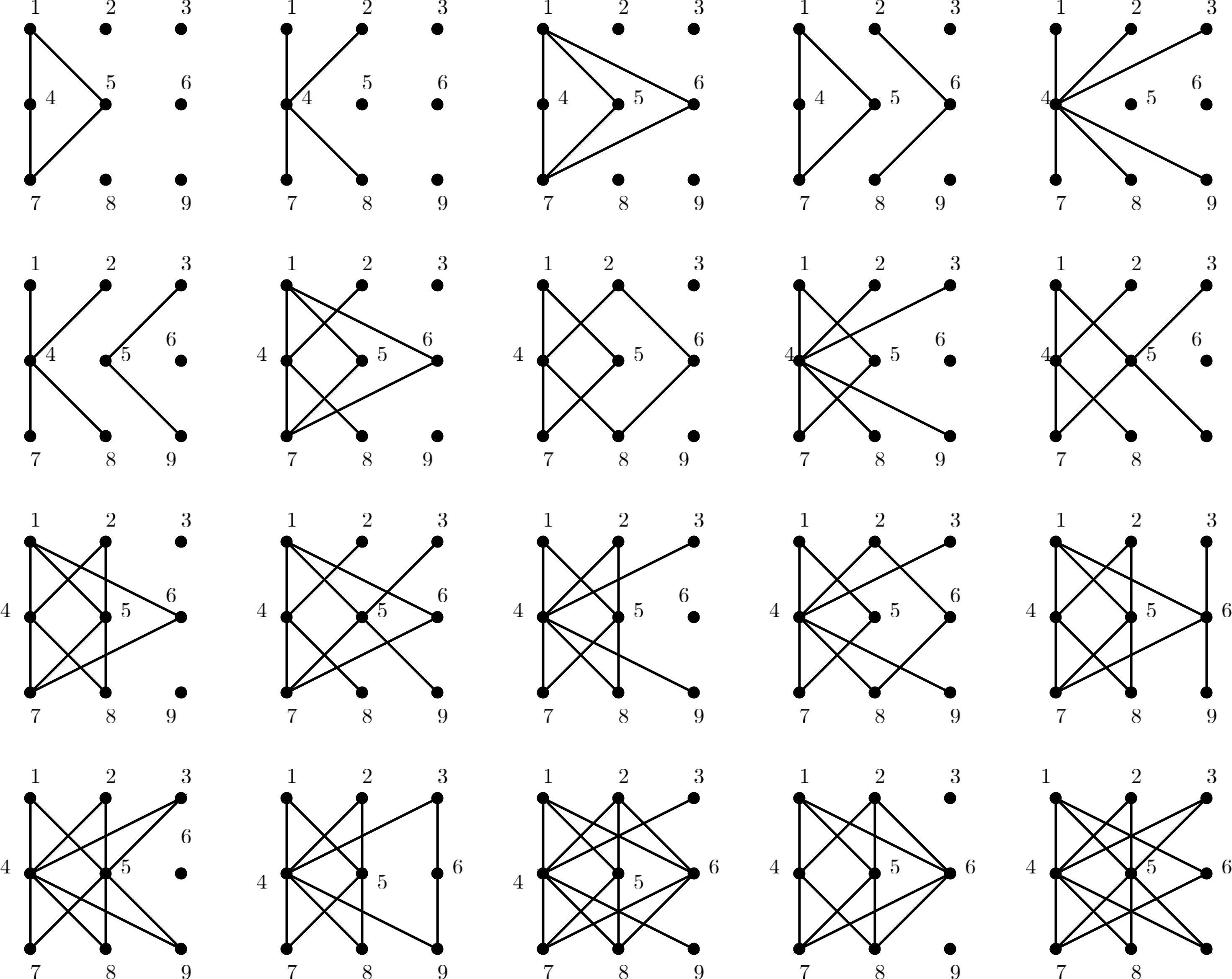}\\
			\cline{2 - 5}
			& 3 & 4 & 250 & \\
			& & & & \includegraphics[scale = .09]{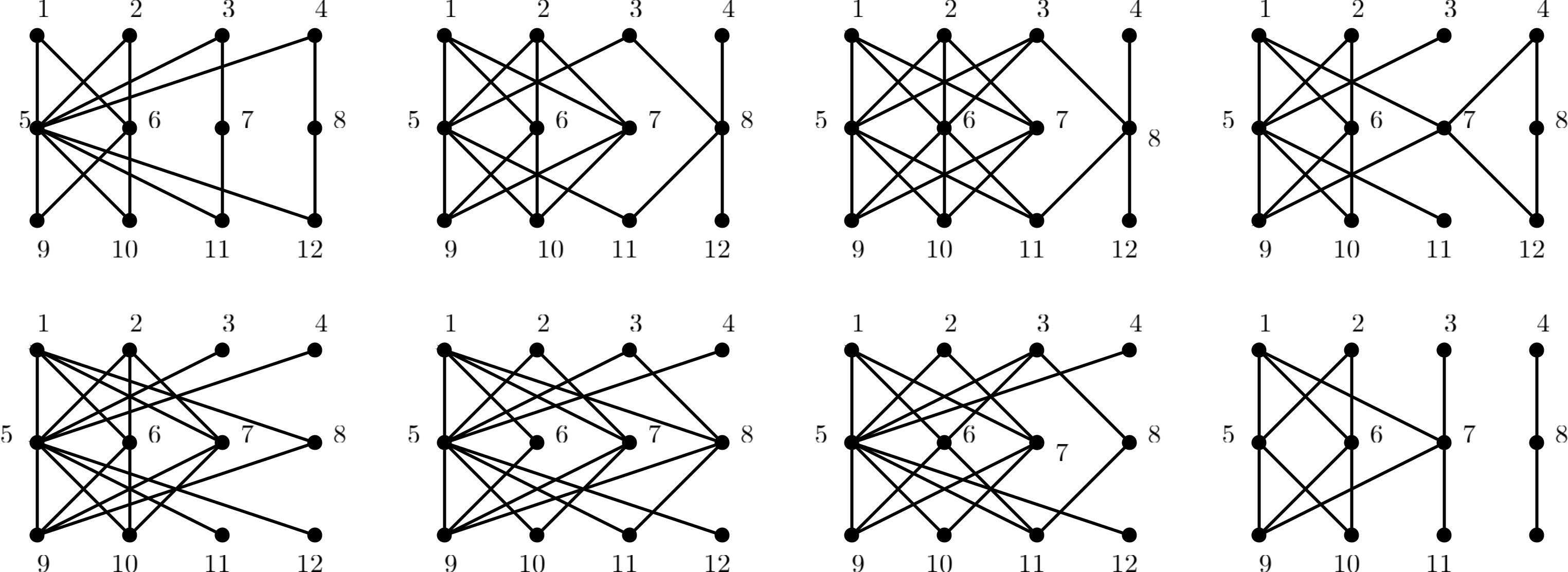}\\
			\hline
			\hline  
			3b & 2 & 2 & 0 & No non-isomorphic GTPT cospectral graph\\
			\cline{2 - 5}
			& 3 & 3 & 4 & \\
			& & & & \includegraphics[scale = .121]{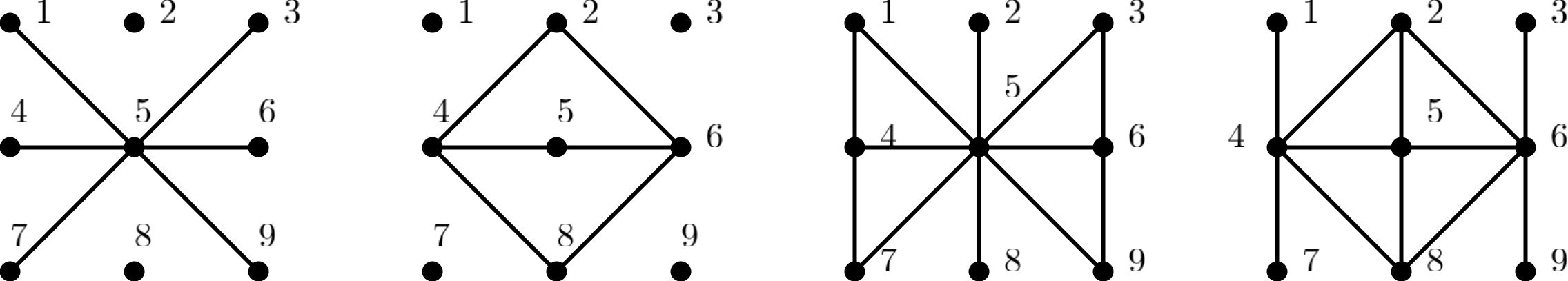}\\
			\cline{2 - 5}
			& 3 & 4 & 10 & \\
			&  & & & \includegraphics[scale = .093]{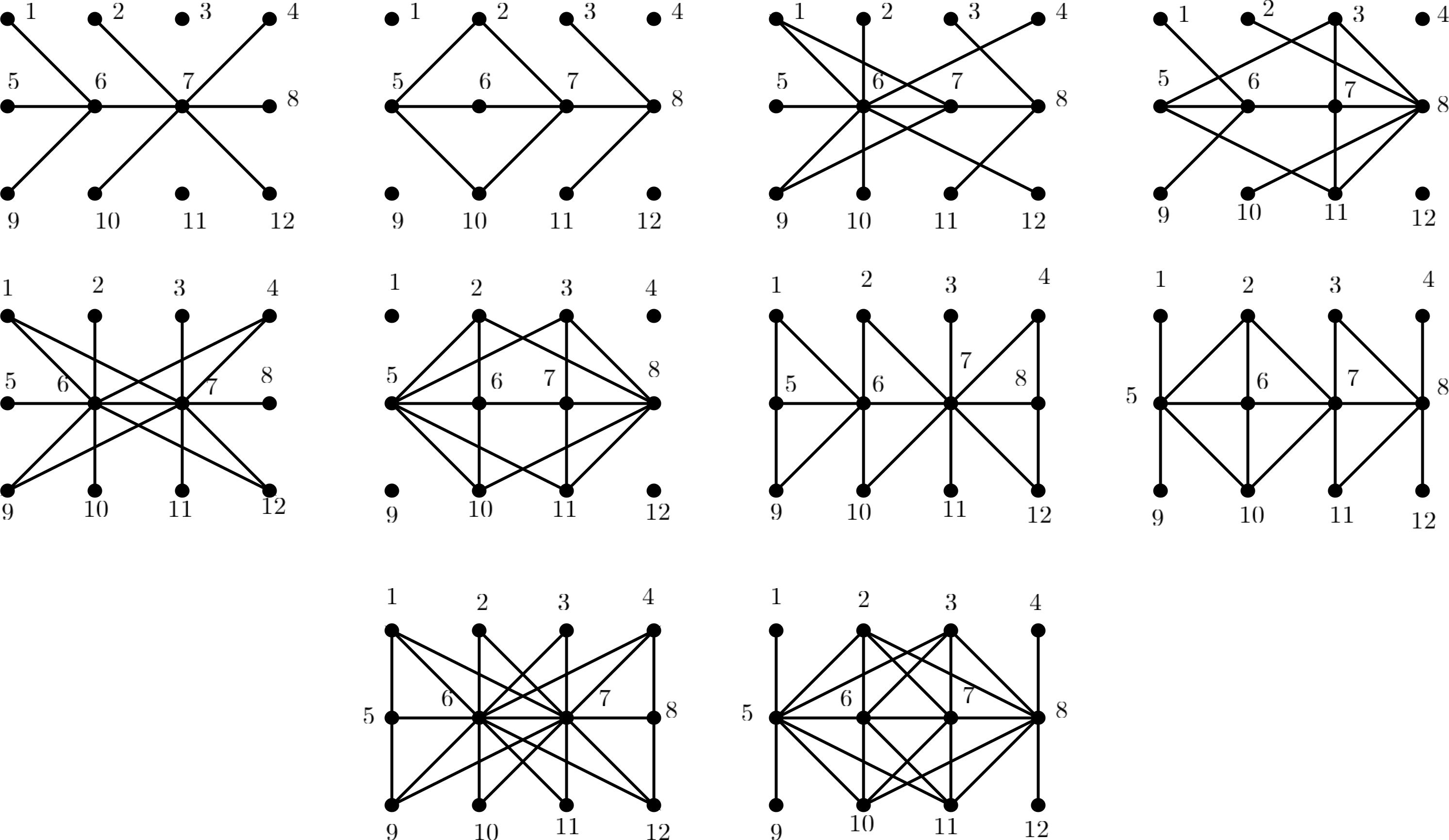}\\
			\hline
			\hline 
			3c & 3 & 2 & 0 & No non-isomorphic GTPT cospectral graph\\
			\cline{2 - 5}
			& 3 & 3 & 0 & No non-isomorphic GTPT cospectral graph\\
			\cline{2 - 5}
			& 3 & 4 & 5 & \\
			& & & & \includegraphics[scale = .123]{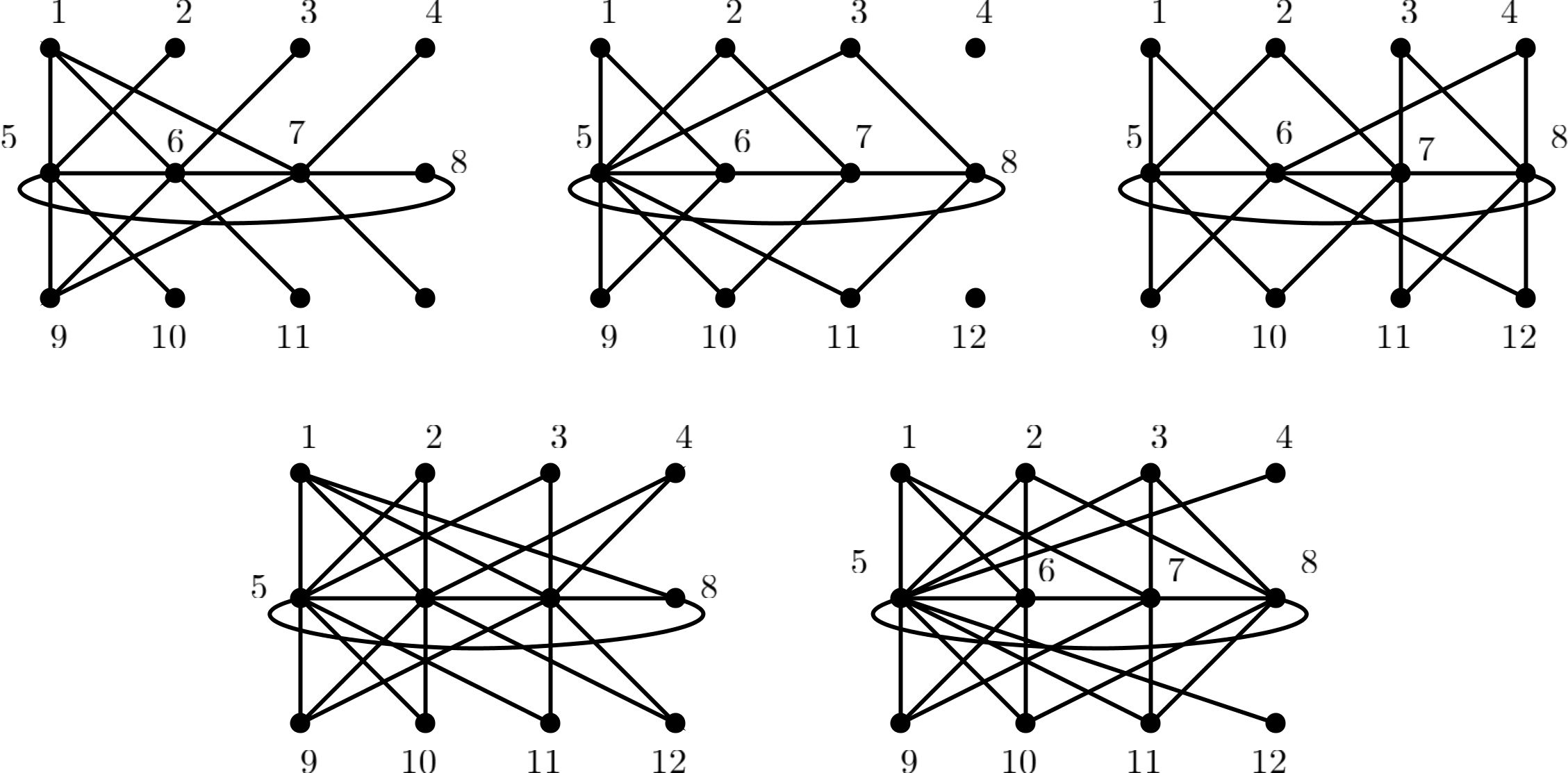}\\
			\hline
		\end{longtable}
		
		Examples by applying combinations of Procedure 1,2, alternate clustering and/or adding new clusters. 
		
		\begin{enumerate}[label = Model \arabic*.,itemindent=*]
			\item[4a.] 
			Consider a clustered graph $G$ with $n=2, m\in\{2,3,4\}$ with clusters $C_1, C_2$ such that $C_1$ is with no edges, $C_2=\mathcal{P}_m,$ and arbitrary edges between the vertices of $C_1$ and $C_2.$ Then we generate the alternating clustered graph $G_a$ with clusters $C_1',\dots, C_m'.$ Following the Procedure 2 we add one vertex $v_{j}, j=1,\dots,m$ to each of these new clusters and create edges. Finally we link the vertices $v_j, j=1,\dots,m$ to form a path $\mathcal{P}_m,$ and hence form the graph $H$ such that it is cospectral with $H^\tau.$ 
			
			\item[4b.]  
			Consider a clustered graph $G$ with $n=2, m\in\{2,3,4\}$ with clusters $C_1, C_2$ such that $C_1=\mathcal{P}_m=C_2$ and arbitrary edges between the vertices of $C_1$ and $C_2.$ Then we generate the alternating clustered graph $G_a$ with clusters $C_1',\dots, C_m'.$ Following the Procedure 2 we add one vertex $v_{j}, j=1,\dots,m$ to each of these new clusters and create edges. Finally we link the vertices $v_j, j=1,\dots,m$ to form a path $\mathcal{P}_m,$ and hence form the graph $H$ such that it is cospectral with $H^\tau.$ 
			
			\item[4c.]  
			Consider a clustered graph $G$ with $n=2, m\in\{2,3,4\}$ with clusters $C_1, C_2$ such that $C_1$ is with no edges, $C_2=\mathcal{C}_m,$ and arbitrary edges between the vertices of $C_1$ and $C_2.$ Then we generate the alternating clustered graph $G_a$ with clusters $C_1',\dots, C_m'.$ Following the Procedure 2 we add one vertex $v_{j}, j=1,\dots,m$ to each of these new clusters and create edges. Finally we link the vertices $v_j, j=1,\dots,m$ to form a path $\mathcal{C}_m,$ and hence form the graph $H$ such that it is cospectral with $H^\tau.$ 
			
		\end{enumerate}
		
		\begin{longtable}{|p{.02\textwidth}|p{.020\textwidth}|p{.020\textwidth}|p{.04\textwidth}|p{.80\textwidth}|}
			\hline
			\rotatebox[origin=c]{90}{Model number} & $m$ & $n$ & $\kappa$ & Examples \\
			\hline
			\hline 
			4a & 2 & 2 & 0 & No non-isomorphic GTPT co-spectral graph\\
			\cline{2 - 5}
			& 3 & 3  & 2 & \\
			& & & & \includegraphics[scale = .12]{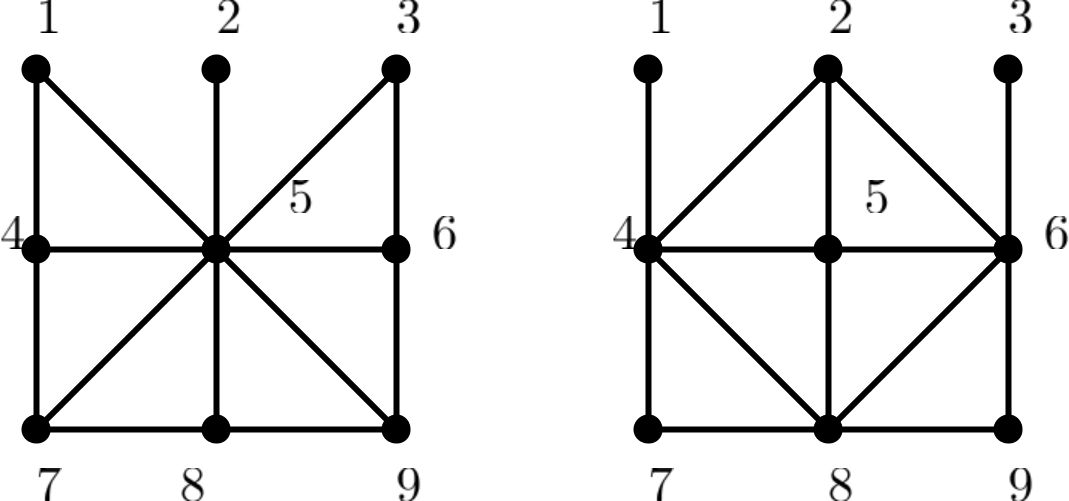}\\
			\cline{2 - 5}
			& 3 & 4 & 7 &  \\
			&  & & & \includegraphics[scale = .091]{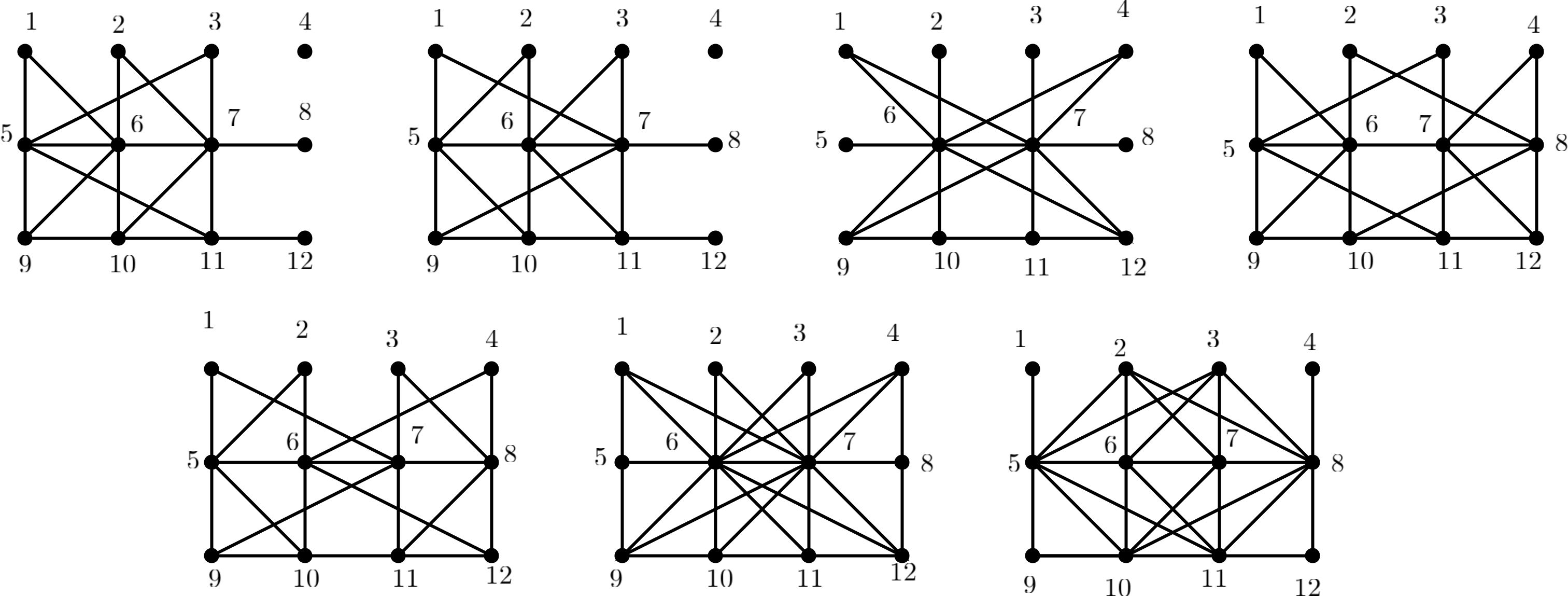}\\
			\hline
			\hline 
			4b & 3 & 2 & 2 &  \\
			& & & & \includegraphics[scale = .09]{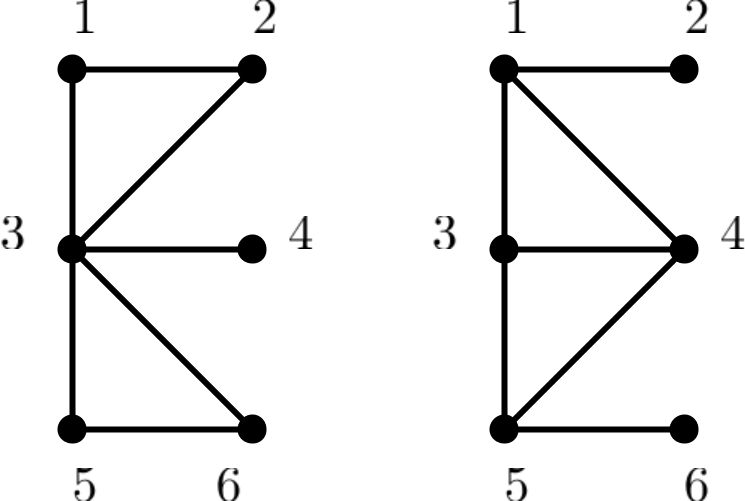}\\
			\cline{2 - 5}
			& 3 & 3 & 113 & \\
			& & & & \includegraphics[scale = .12]{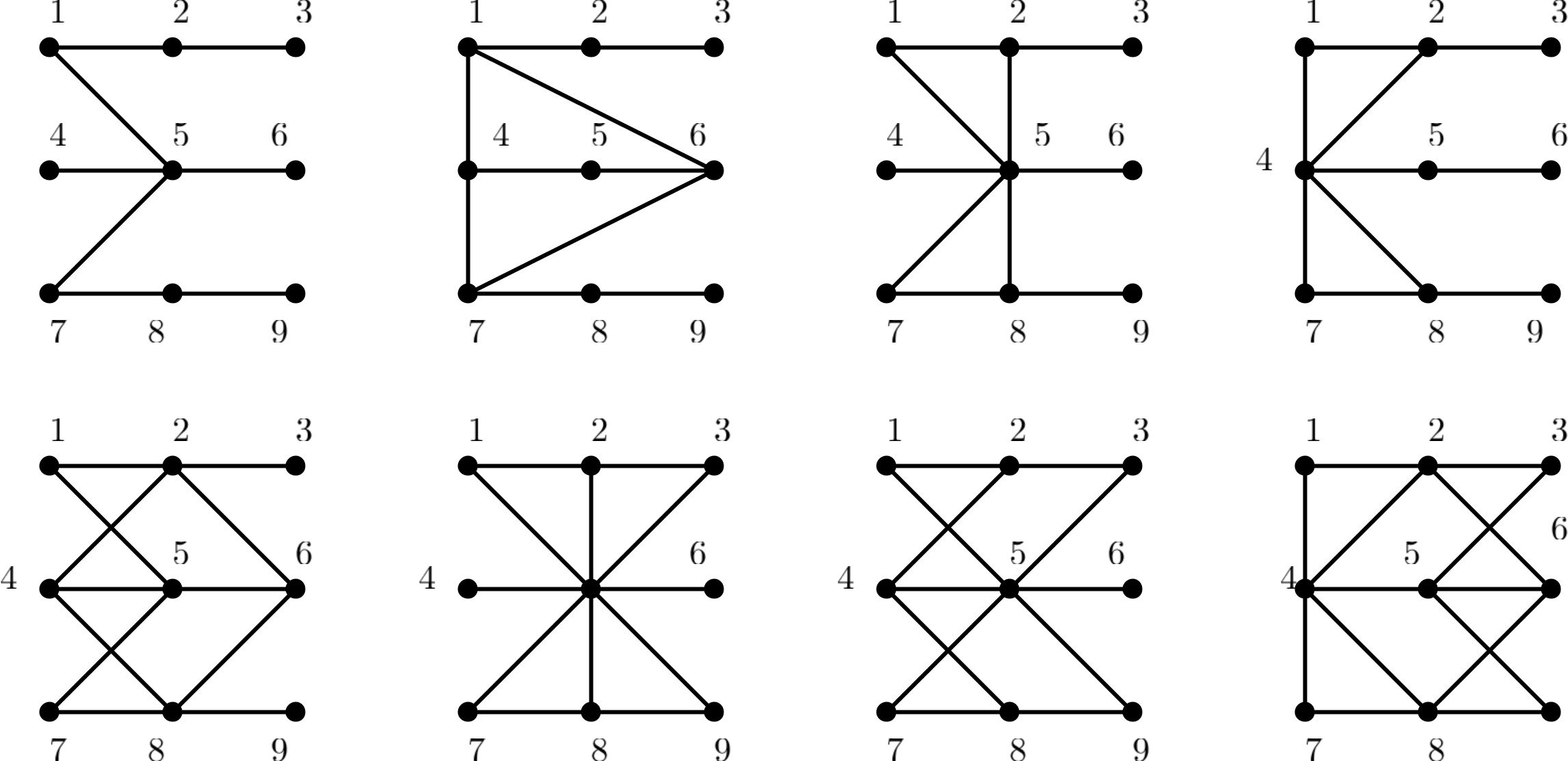}\\
			\hline
			\hline 
			4c & 2 & & 0 & No non-isomorphic GTPT cospectral graph\\
			\cline{2 - 5}
			& 3 & & 0 & No non-isomorphic GTPT cospectral graph\\
			\cline{2 - 5}
			& 3 & 4 & 8 & \\
			& & & & \includegraphics[scale = .091]{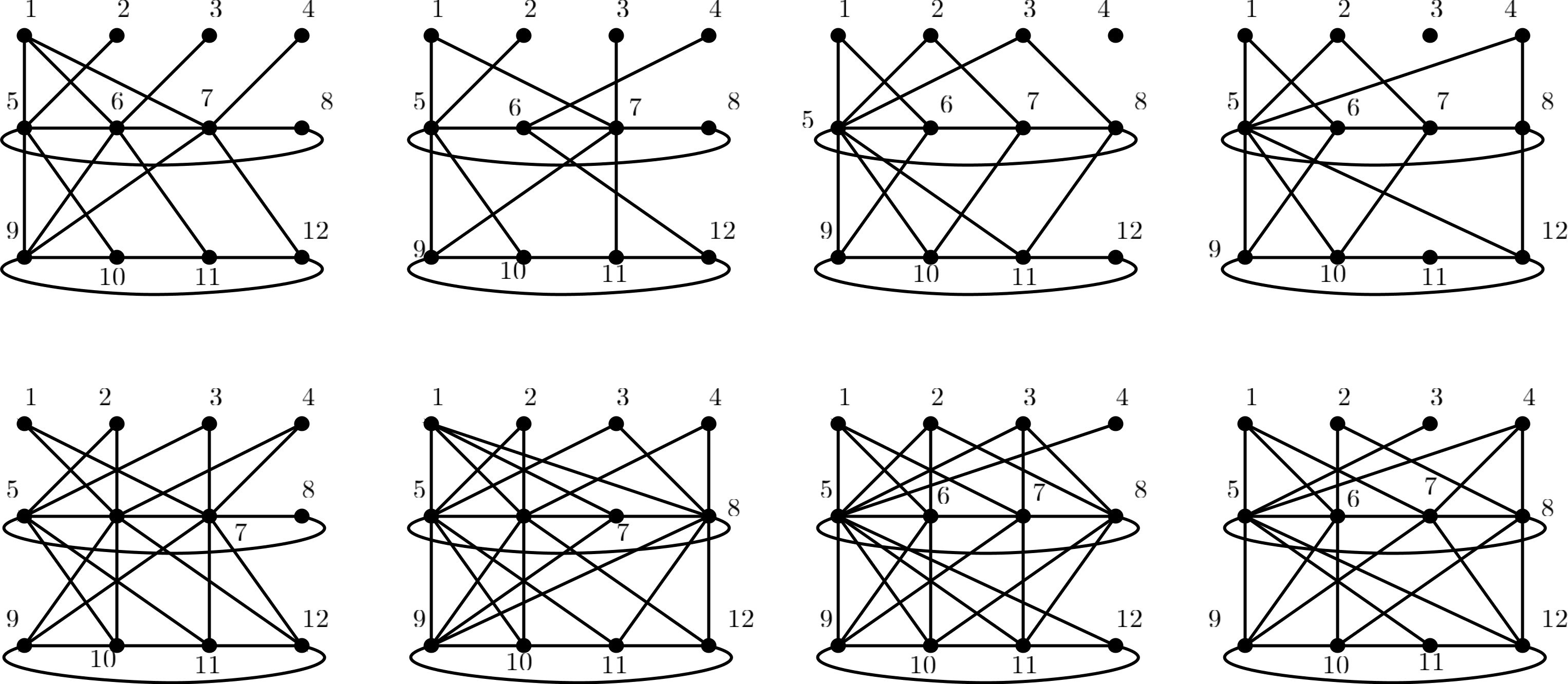}\\
			\hline
		\end{longtable}

		\section{Conclusion and future work}
		
		We have identified classes of graphs for which the corresponding GTPT equivalent graphs are cospectral and hence it is established that GTPT approach can act as a novel method for construction of cospectral mates of a given graph. In contrast to the exisiting combinatorial approches thus we introduce a matrix function approch (known as partial transpose of a block matrix) for construction of cospectral graphs. This opens up a new idea of treating a graph on a composite number of vertices as a clustered graph and identifying the properties of its clusters that guarantees to obtain cospectral mates of the given graph. We also developed two procedures to construct small and/or big GTPT euivalent cospectral graphs from a given GTPT euivalent cospectral graph. Finally we produce several examples of non-isomorphic cospectral graphs using the GTPT approch and the procedures proposed in the paper.
		
		We mention that the GTPT approch can be extended in many directions, in particular, for weighted graphs and construction of Laplacian co-spectral graphs which we would like to investigate in future. Becides, it would be an interesting problem to determine propertics of graphs that are invariant under GTPT.

	\section*{Funding}
		
		SD was supported by a doctoral fellowship provided by MHRD, Government of India.


\end{document}